\documentclass[10pt]{amsart}

\usepackage[utf8]{inputenc}
\usepackage{color}
\pdfoutput=1
\usepackage{geometry} 
\usepackage{graphicx,url}
\usepackage{amssymb,amsmath, amsthm}
\usepackage{epstopdf}
\usepackage{enumitem}
\usepackage{fancyhdr}
\usepackage{multirow}
\usepackage{tabularx}
\usepackage{subcaption}
\usepackage{longtable}
\usepackage{overpic}

\newcommand{\re}{\mathbb{R}}

\newcommand{\cor}{\mathcal}

\usepackage{hyperref}
\usepackage{graphicx,url}

\newtheorem{theorem}{Theorem}

\newtheorem{remark}{Remark}[section]

\newtheorem{definition}{Definition}

\title[Regularized Bridge-type estimation with multiple penalties]{Regularized Bridge-type estimation with multiple penalties }
\numberwithin{equation}{section}

\newcommand{\sgn}{\text{sgn}}
\newcommand{\de}{\mathrm{d}}

\usepackage{xcolor}

\allowdisplaybreaks
\date{\today}

\author{Alessandro De Gregorio, Francesco Iafrate}
\address{Department of Statistical Sciences, ``Sapienza" University of Rome,
	P.le Aldo Moro, 5 - 00185, Rome, Italy}
\email{alessandro.degregorio@uniroma1.it}
\email{francesco.iafrate@uniroma1.it}

\begin{document}
	
\maketitle

\begin{abstract} 
The aim of this paper is to introduce an adaptive penalized estimator for identifying the true reduced parametric model under the sparsity assumption. In particular, we deal with the framework where the unpenalized estimator of the structural parameters  needs  simultaneously multiple rates of convergence (i.e. the so-called mixed-rates asymptotic behavior). We introduce a Bridge-type estimator by taking into account penalty functions involving $\ell^q$ norms $(0<q\leq 1)$. We prove that the proposed regularized estimator satisfies the oracle properties. 

Our approach is useful for the estimation of stochastic differential equations in the parametric sparse setting. More precisely, under the high frequency observation scheme, we apply our methodology to an ergodic diffusion and introduce a procedure for the selection of the tuning parameters. Furthermore, the paper contains a simulation study as well as a real data prediction in order to assess about the performance of the proposed Bridge estimator.

	\end{abstract}

	
\section{Introduction}
	
	Statistical learning ensures high prediction accuracy and discovers relevant predictive variables. Furthermore, variable selection is particularly important when the true underlying model has a sparse representation. Identifying significant variables will enhance the prediction performance of the fitted model. A possible way to address this issue is represented by the stepwise and subset selection procedures. Nevertheless, the main drawbacks of this approach are the computational complexity and the variability.  In the last twenty years  the penalized statistical methods became very popular in the variable selection framework (see, e.g., \cite{hastie1} and \cite{hastie2}).	

Let us consider the classical linear regression model $y= {\bf X} \theta + \varepsilon,$ where $y$ is the response vector and ${\bf X}$ is the $n\times p$ predictor matrix of standardized variables, $\theta:=(\theta_1,...,\theta_p)'\in\mathbb R^p$ is the parametric vector and $\varepsilon$ is a Gaussian vector with independent components with mean zero. 
The least absolute shrinkage and selection operator (LASSO), introduced in \cite{tib}, is a useful and well studied approach to the problem of model selection. Its major advantage is the simultaneous execution of both parameter estimation and variable selection. The LASSO estimator is obtained by solving the $\ell_1$ penalized least squares problem 
\begin{equation}
\hat \theta(\text{LASSO})=\arg\min_{\theta}\{|y-{\bf X}\theta|^2+\lambda||\theta||_1\}
\end{equation}
 where $\lambda\geq 0$ (tuning parameter), $|\cdot|$ is the euclidean distance and $||\theta||_1=\sum_{j=1}^p |\theta_j|$, or, equivalently, by dealing with an unpenalized optimization problem with a constraint
\begin{equation} 
\hat \theta(\text{LASSO})=\arg\min_{\theta}\{|y-{\bf X}\theta|^2\},\quad \text{subject to }\quad  ||\theta||_1\leq t,
\end{equation}
with $t\ge 0.$ Let us notice that the $\ell_1$ penalty admits some singularities. For this reason some coefficients are shrunken to zero.

The Bridge estimation generalized the LASSO approach by using $\ell^q$ norm (see, e.g., \cite{frank} and \cite{knight}) as follows
\begin{equation}\label{eq:bridge}
\hat \theta(\text{Bridge})=\arg\min_{\theta}\left\{|y-{\bf X}\theta|^2+\lambda||\theta||_q\right\},
\end{equation}
where $||\theta||_q=\sum_{j=1}^p |\theta_j|^q, q>0.$ The estimator \eqref{eq:bridge} becomes LASSO for $q=1,$ and Ridge for $q=2.$ 
Notice that, AIC or BIC criterion can be viewed as limiting cases of Bridge estimation as $q\to 0;$ i.e.
$$\lim_{q\to 0}\sum_{j=1}^p |\theta_j|^q=\sum_{j=1}^p 1_{\theta_j\neq 0}.$$ 
There are other possible approaches: for instance, in \cite{fan1}, the authors proposed a shrinkage procedure based on the smoothly clipped deviation (SCAD) penalty term.

The reguralization methods could allow the dimensionality of the parameter space to change with the sample size, this is the main advantage of the LASSO approach over the classical information criterions (AIC, BIC, etc.) which use a fixed penalty on the size of a model.

As argued in \cite{fan1}, \cite{peng} and \cite{fan2}, a good selection procedure should have (asymptotically) the so-called ``oracle'' properties:
\begin{itemize}
\item[i)]  consistently estimates null parameters as zero; i.e the selection procedure identifies the right subset model;
\item[ii)]  has the optimal estimation rate and converges to a Gaussian random variable $N(0,\Sigma_0)$ where $\Sigma_0$ is the covariance matrix of the true subset model.
\end{itemize} 
 As shown in \cite{zou}, since LASSO procedure assigns the same amount of penalization to each parameter, it does not represent an oracle procedure. For this reason the author introduced the following adaptive LASSO estimator
\begin{equation*}
\hat \theta_n(\text{AdaLASSO})=\arg\min_{\theta}\left\{|y-{\bf X}\theta|^2+\lambda_n\sum_{j=1}^p \hat w_j| \theta_j|\right\},
\end{equation*}
 where $\lambda_n>0$ and $(\hat w_j)_{j=1}^p$ are suitable data-driven  weights. Usually $\hat w_j=1/|\hat\theta_j(\text{ols})|^\gamma, \gamma>0.$ As the sample size grows, the weights for zero-coefficient regressors increase to infinity, whereas the weights for nonzero-coefficient predictors tend to a finite constant. Thus, we are able to estimate consistently null parameters and $\hat \theta_n(\text{AdaLASSO})$ is oracle.

 Originally, the LASSO procedure was introduced for linear regression problems, but, in the recent years, this approach has been applied to different fields of stochastic processes. In \cite{wang2}  the problem of shrinkage estimation of regressive and autoregressive coefficients has been dealt with, while in \cite{nardi} penalized order selection AR($p$) models is studied. Furthermore, in \cite{caner} is shown that the Bridge estimator can be used to differentiate stationarity from unit root type of non stationarity and to select the optimal lag in AR series as well.  For other issues on penalized estimation problems for time series analysis, the reader can also consult \cite{basu}. 
 
 Very recently,  regularized estimators has been applied to multidimensional diffusion processes and point processes, and represent a new research topic in the field of statistics for stochastic processes. For instance, in the high frequency framework, the reader can consult \cite{deg}, \cite{mas}, \cite{suz}, where the authors used penalized selection procedure for discovering the underlying true model.  In \cite{deg2}  and \cite{gai}, LASSO and Bridge estimators have been applied to continuously observed stochastic differential equations.
 
 In this paper, we address the estimation problem for a sparse parametric model where different rates of convergence must be considered simultaneously for the asymptotic identification of the vector of structural parameters (i.e. mixed-rates asymptotics). In particular, we introduce a Bridge-type estimator  by means of the least squares approximation approach developed in \cite{wang1}. The main idea is represented by an objective function with multiple adaptive penalty functions involving $\ell^q$ norms $(0<q\leq 1).$  We will show that the estimator obtained by minimizing the above objective function satisfies the oracle properties.

Our motivating example is represented by a multidimensional ergodic diffusion process solution to a stochastic differential equation; i.e.
\begin{equation*}
	\de X_t=b(X_{t},\alpha)\de t+\sigma(X_{t},\beta)\de W_t,\quad X_0=x_0.
	\end{equation*}
	Under the high frequency observation scheme, $\alpha\in\mathbb R^{p_1}$ and $\beta\in\mathbb R^{p_2},$ admit optimal estimators $\hat\alpha_n$ and $\hat \beta_n$ having two different asymptotic rates (i.e. infill asymptotic).  We will assume the sparsity condition for the above diffusion process; i.e. some coefficients of the true values of $\alpha$ and $\beta$ are exactly equal to zero. We will apply our methodology to this framework. Actually, it would be possible to deal with other types of random models described by stochastic differential equations: for instance, small diffusions (see, e.g., \cite{sor} and \cite{glo}), and diffusions with jumps (see, e.g., \cite{shim}, \cite{clem} and \cite{mas2}). Such as random processes define parametric models with two groups of parameters having two different asymptotic rates.

Furthermore, there are several econometric models where the exact evaluation of the structural parameters requires estimators with mixed-rates asymptotics. For instance, in \cite{bertille} the asymptotic theory of GMM inference is extended; the main goal of this work is to allow sample counterparts of the estimating equations to converge at (multiple) rates, different from the usual square-root of the sample size. Moreover, \cite{bertille} contains some econometric examples where the mixed-rates behavior arises. In \cite{lee} the theoretical properties of the maximum likelihood estimator
and the quasi-maximum likelihood estimator for the spatial autoregressive model are investigated. The
rates of convergence of those estimators may depend on some general features of the
spatial weights matrix of the model. When each unit can be influenced by many neighbors, irregularity of the information matrix may occur and various components of the estimators may have different rates of convergence.

The paper is organized as follows. In Section \ref{sec2}, we introduce the estimation parametric problem. The setup is that of a parametric model with unpenalized estimators that asymptotically behave well under multiple rates of convergence. In this setting, we introduce a regularized estimator by means of the least squares approximation approach developed in \cite{wang1}. Section \ref{sec3} contains the discussion of the oracle properties of the introduced estimator. Section \ref{sec4} is devoted to the application of our methodology to diffusion processes related to stochastic differential equations. Furthermore, a selection procedure for tuning parameters, based on the standardized residuals of the discretized sample path, is proposed. In order to evaluate the performance of our estimator, in Section \ref{numres} a simulation study on a linear diffusion process is carry on to select the true underlying model.  In the same section we  test the prediction accuracy of our methodology  for four financial time series of daily closing stock prices of major tech companies: Google, Amazon, Apple and Microsoft. Besides, in the framework of ergodic diffusions, we compare the Bridge estimator introduced in this paper and the disjoint method developed in \cite{suz}. All the proofs of the oracle properties are collected in the last section.

\section{Adaptive Bridge-type estimation with multiple penalties}\label{sec2}
	
	Let us introduce the shrinking estimator in a general setup. We deal with a parameter of interest $\theta:=(\theta^1,...,\theta^m)',$ where $\theta^i:=(\theta_1^i,...,\theta_{p_i}^i), p_i\in\mathbb N,i=1,...,m.$ Furthermore, $\theta\in \Theta:=\Theta_1\times\cdots\times \Theta_m\subset \mathbb R^{\mathfrak{p}}, \mathfrak{p}:=\sum_{i=1}^mp_i,$ where $\Theta_i$ is a bounded convex subset of $\mathbb R^{p_i}.$ We denote by $\theta_0:=(\theta_0^1,...,\theta_0^m)',$ where $\theta_0^i:=(\theta_{0,1}^i,...,\theta_{0,p_i}^i), i=1,...,m,$ the true value of $\theta.$ 
	
	Assume that there exists  a loss function
	 $\theta\mapsto \mathfrak L_n(\theta)$ and $$\tilde\theta_n=(\tilde\theta_n^1,...,\tilde\theta_n^m)'\in\arg\min_\theta \mathfrak L_n(\theta).$$
	 Usually $\mathfrak L_n(\theta)$ is a (negative) log-likelihood function or the sum of squared residuals. Furthermore, suppose that $\tilde\theta_n$ admits a mixed-rates asymptotic behavior in the sense of \cite{rad}; that is for the asymptotic estimation of $\theta_0^i, i=1,...,m,$ is necessary to consider simultaneously different rates of convergence for $\tilde\theta_n^i, i=1,...,m$.

	We assume that $\theta_0$ is sparse (i.e., some components of $\theta_0$ are exactly zero). Let $p_i^0:=|\{j:\theta_{0,j}^i\neq 0\}|, i=1,...,m,$ and $\mathfrak p^0:=\sum_{i=1}^mp_i^0.$ Therefore, our target is the identification of the true model $\theta_0$ by exploiting a multidimensional random sample $(X_n)_{n}$ on the probability space $(\Omega,\mathcal F, P).$  
	
	In order to carry out simultaneously estimation and variable selection, we use a penalized approach involving suitable shrinking terms. Since we have to take into account the multiple asymptotic behavior of the non-regularized estimator $\tilde\theta_n$, we suggest to  penalize different sets of parameters with different norms. Therefore, the adaptive objective function with weighted $\ell^{q_i}$ penalties should be given by
	\begin{align}\label{eq:stLASSO}
	&\mathfrak L_n(\theta)+\left[\sum_{j=1}^{p_1}\lambda_{n,j}^1|\theta_j^1|^{q_1}+...+\sum_{j=1}^{p_m}\lambda_{n,j}^m|\theta_j^m|^{q_m}\right],\quad q_i\in(0,1], i=1,...,m,
	\end{align}
	where $(\lambda_{n,j}^i)_{n\geq 1}, j=1,...,p_i,i=1,...,m,$ are sequences of real positive random variable representing an adaptive amount of the shrinkage for  each element of $\theta^i.$
	The Bridge-type estimator is the minimizer of the objective function \eqref{eq:stLASSO}, which reduce to the LASSO-type estimator if $q_i=1$ for any $i$. This is a non-linear optimization problem which might be numerically challenging to solve. By resorting the  least squares approximation approach developed in \cite{wang1} and \cite{suz}, we can replace \eqref{eq:stLASSO} with a more tractable objective function. Indeed, if $\mathfrak L_n$ is twice differentiable with respect to $\theta,$ we have
	\begin{align*}\mathfrak L_n(\theta)&\simeq \mathfrak L_n(\tilde\theta_n)+(\theta-\tilde\theta_n)' \nabla_\theta \mathfrak L_n(\tilde\theta_n)+\frac12(\theta-\tilde\theta_n)'  \ddot{\mathfrak L}_n(\tilde\theta_n)(\theta-\tilde\theta_n)\\
	&=\mathfrak L_n(\tilde\theta_n)+\frac12(\theta-\tilde\theta_n)'  \ddot{\mathfrak L}_n(\tilde\theta_n)(\theta-\tilde\theta_n),
	\end{align*}
	where $\ddot{\mathfrak L}_n$ represents the Hessian matrix.
	Therefore we may minimize instead of \eqref{eq:stLASSO} the following objective function with multiple penalty terms 
	\begin{equation}\label{eq:stLASSO2}
	(\theta-\tilde\theta_n)'  \ddot{\mathfrak L}_n(\tilde\theta_n)(\theta-\tilde\theta_n)+\left[\sum_{j=1}^{p_1}\lambda_{n,j}^1|\theta_j^1|^{q_1}+...+\sum_{j=1}^{p_m}\lambda_{n,j}^m|\theta_j^m|^{q_m}\right].
	\end{equation}
	The gain of \eqref{eq:stLASSO2} is twofold: it reduces the computational complexity of \eqref{eq:stLASSO}; furthermore the  least squares term allows to unify many different types of penalized objective functions.

	Now, inspired by \eqref{eq:stLASSO2}, we are able to define the adaptive penalized estimator studied in this paper.
	
	\begin{definition} Let $\hat G_n$ be a $\mathfrak{p}\times \mathfrak{p}$  almost surely positive definite symmetric random matrix depending on $n.$ We define the adaptive Bridge-type estimator $\hat \theta_n:\mathbb{R}^{(n+1)\times d}\to\overline \Theta$ as follows 
	
	\begin{equation}\label{eq:LASSOtypest}
	\hat{\theta}_n=(\hat\theta_n^{1},...,\hat\theta_n^{m})'\in\arg\min_{\theta\in\overline\Theta}\mathcal{F}_n(\theta)
	\end{equation}
	where
	\begin{equation}\label{eq:LASSOtype}
	\mathcal{F}_n(\theta):=(\theta-\tilde{\theta}_n)'\hat G_n(\theta-\tilde{\theta}_n)+\left[\sum_{j=1}^{p_1}\lambda_{n,j}^1|\theta_j^1|^{q_1}+...+\sum_{j=1}^{p_m}\lambda_{n,j}^m|\theta_j^m|^{q_m}\right],
	\end{equation}
with $q_i\in(0,1], i=1,...,m.$
	\end{definition}
	Clearly, \eqref{eq:LASSOtype} reduces to \eqref{eq:stLASSO2} if $\hat G_n:=\ddot{ \mathfrak L}_n(\tilde\theta_n).$ 
	
	The estimator \eqref{eq:LASSOtypest} coincides with the estimator introduced in \cite{suz} for $m=1$ (actually, they are slightly different because we will require different asymptotic conditions on the matrix $\hat G_n$).
	Our scope is to generalize the approach developed in \cite{suz}, in order to extend the Bridge-type methodology to statistical parametric models with multiple rates of convergence. Thus, for this reason, the objective function \eqref{eq:LASSOtype} involves different norms, one for each set of parameters.
	For instance, in the case of ergodic diffusions the shrinking estimator \eqref{eq:LASSOtypest} is theoretical equivalent (see Theorem \ref{th:sde} below) to its counterpart studied in \cite{suz}, Section 7.1. Furthermore, when we apply the Bridge-type estimation procedure, it is necessary to work in the finite sample size setting. Therefore, our methodology, based on the joint estimation, is able to take into account the cross-correlations between the  variables of the model, by means of the random matrix $\hat G_n.$ This last issue is a crucial point in the statistical learning, where the correct identification of the dependent variables improves the performance of the fitted model. These features could be lost if we split the penalized estimation of the parameters.

\section{ Oracle properties}\label{sec3}
	For the sake of simplicity, hereafter, we assume $\theta_{0,j}^i\neq 0, j=1,...,p_i^0,$ for any $i=1,...,m.$   We deal with $r_n^i,i=1,...,m,$ representing sequences of positive numbers tending to 0 as $n\to\infty.$ ${\bf I}_m$ stands for the identity matrix of size $m$. Furthermore, we introduce the following matrices $$ A_n:= \text{diag}(r_n^1 {\bf I}_{p_1},...,r_n^m{\bf I}_{p_m}).$$

	We main assumptions in the paper are the following ones.
	
	\begin{itemize}
		\item[A1.] Let $\hat{\mathfrak D}_n:=A_n \hat G_n A_n$.   There exists a $\mathfrak{p}\times \mathfrak{p}$ positive definite symmetric random matrix $G$ such that $$\hat{\mathfrak D}_n\stackrel{p}{\longrightarrow}G.$$
		\item[A2.] The estimator $\tilde \theta_n$ is consistent; i.e. $$A_n^{-1}(\tilde{\theta}_n-\theta_0)=\left(\frac{1}{r_n^{1}}(\tilde\theta_n^1- \theta_0^1),...,\frac{1}{r_n^{m}}(\tilde\theta_n^m- \theta_0^m)\right)'=O_p(1).$$
		\item[A3.] The estimator $\tilde \theta_n$ is asymptotically normal; i.e. $$A_n^{-1}(\tilde{\theta}_n-\theta_0)=\left(\frac{1}{r_n^{1}}(\tilde\theta_n^1- \theta_0^1),...,\frac{1}{r_n^{m}}(\tilde\theta_n^m- \theta_0^m)\right)'\stackrel{d}{\longrightarrow} N_{\mathfrak p}(0,\mathfrak I),$$
		where $\mathfrak I:=\Gamma^{-1}$ and $\Gamma$ is a $\mathfrak{p}\times \mathfrak{p}$ positive definite symmetric matrix. 
	\end{itemize}
	
	The conditions A2 and A3 reveals the mixed-rates asymptotic behavior of the estimator $\tilde{\theta}_n.$ Actually, A3 could be replaced with a stronger condition involving the stable convergence to a mixed normal random variable (see, e.g., \cite{suz}).
	
	Let us denote by
	 $a_n^i:=\max\{\lambda_{n,j}^i, j\leq p_i^0\},$ $b_n^i:=\min\{\lambda_{n,j}^i, j> p_i^0\}$; we introduce the following conditions.
	 
\begin{itemize}
		\item[B1.]  $r_n^i a_n^i=O_p(1)$ for any $i=1,...,m.$		
		\item[B2.]  $r_n^i a_n^i=o_p(1)$ for any $i=1,...,m.$
		\item[B3.] $(r_n^i)^{2-q_i} b_n^i\stackrel{p}{\longrightarrow}\infty,$ for any $i=1,...,m.$ 
	\end{itemize}	  
	
	The main goal of this section is to argue on the theoretical features of the regularized statistical procedure arising from Definition 1; i.e. we are able to prove that the estimator $\hat\theta_n$ is asymptotically oracle. 

	\begin{theorem}[Consistency]\label{th:cons}
		Assume A1, A2 and B1, then $$A_n^{-1}(\hat\theta_n-\theta_0)=O_p(1).$$
	\end{theorem}

	For the vectors $x^i=(x_1,...,x_{p_i})'$  we deal with the following notation: $x_{\star}^i:=(x_1,...,x_{p_i^0})',$ $x_{\bullet}^i:=(x_{p_i^0+1},...,x_{p_i})'.$
	
	\begin{theorem}[Selection consistency] \label{th:sc} If the assumptions A1, A2, B1 and B3 are satisfied,  we have that
		$$P(\hat \theta_{n\bullet}^{i}=0)\longrightarrow 1,\quad \,i=1,...,m,$$
		as $n\longrightarrow\infty.$
	\end{theorem}
	
	Theorem \ref{th:sc} allows to claim that with probability tending to 1, all of the zero parameters must be estimated as 0. Theorem \ref{th:cons} leads to the consistency of the estimators of the nonzero coefficients. Both theorems imply that the Bridge estimator \eqref{eq:LASSOtypest}  identifies the true model consistently.

  	In what follows we adopt the following notation:
	let $M$ be a partitioned $p_i\times p_j$ matrix, $1\leq i,j\leq m,$ 
	 $$M =\left(  \begin{matrix} 
	      M_{\star\star} &   M_{\star\bullet} \\
	        M_{\bullet\star} &   M_{\bullet\bullet} \\
	   \end{matrix}	\right),$$  
	   where the blocks are given by:
	   \begin{itemize}
	   \item  $M_{\star\star}=(m_{ij})_{1\leq i\leq p_i^0,1\leq j\leq p_j^0}$ is a $p_i^0\times p_j^0$ matrix;
	   \item $M_{\star\bullet}=(m_{ij})_{1\leq i\leq p_i^0,p_j^0< j\leq p_j},$ is a $p_i^0\times (p_j-p_j^0)$ matrix;
	   \item $M_{\bullet\star}=(m_{ij})_{p_{i}^0< i\leq p_i ,1\leq j\leq p_j^0}$ is a  $(p_i-p_i^0)\times p_j^0$ matrix;
	   \item  $M_{\bullet\bullet}=(m_{ij})_{p_{i}^0< i\leq p_i,p_j^0<j\leq p_j}$ is a  $(p_i-p_i^0)\times(p_j-p_j^0)$ matrix.
	   \end{itemize}
	   
	     Moreover, we take into account the following assumption representing a special case of the condition A1.   
	      \begin{itemize}
		\item[C1.]  There exist $p_i\times p_i$ positive definite symmetric random matrices $G^{ii}, i=1,2,...,m.$ such that $$\hat{\mathfrak D}_n\stackrel{p}{\longrightarrow}G:= \text{diag}(G^{11},G^{22},...,G^{mm}).$$
	\end{itemize}

	   Let us assume C1 and introduce the following $\mathfrak p^0\times \mathfrak p$ matrix
 	   $$\mathfrak G:= \left(  \begin{matrix} 
	          \mathfrak G_1&0&0& \cdots&0\\
	        0 &  \mathfrak G_2& 0& \cdots &0 \\
	             \vdots &\vdots&
\vdots & \vdots&\vdots \\
	     0&0&0&\cdots &       \mathfrak G_m
	      \end{matrix}\right),$$
	      where $  \mathfrak G_i:=({\bf I}_{p_i^0}\,\,\, ( G_{\star\star}^{ii})^{-1} G^{ii}_{\star\bullet}), i=1,2,...,m.$ 
	      	      Now, we are able to prove the asymptotic normality of the Bridge-type estimator and its efficiency with respect to the true subset model.

	\begin{theorem}[Asymptotic normality]\label{th:an}
		Under the assumptions C1,A2, B2 and B3, we have that
		\begin{equation}\label{oraclegen}\left(\frac{1}{r_n^1}(\hat\theta_n^1-\theta_0^1)_\star,...,\frac{1}{r_n^m}(\hat\theta_n^m-\theta_0^m)_\star\right)'-   \mathfrak G\left\{A_n^{-1}(\tilde\theta_n-\theta_0)\right\}  \stackrel{p}{\longrightarrow}0,\end{equation}
		as $n\longrightarrow\infty.$
		Furthermore, adding A3, we obtain 
				\begin{equation}\label{oracleasnorm}
				\left(\frac{1}{r_n^1}(\hat\theta_n^1-\theta_0^1)_\star,...,\frac{1}{r_n^m}(\hat\theta_n^m-\theta_0^m)_\star\right)' \stackrel{d}{\longrightarrow}N_{\mathfrak p^0}\left(0, \mathfrak G\,\mathfrak I\,\mathfrak G'\right),
				\end{equation}
				as $n\longrightarrow\infty,$ and if $G=\Gamma,$  one has that
				$$\mathfrak G\,\mathfrak I\,\mathfrak G':=\text{diag}\left((\Gamma_{\star\star}^{11})^{-1},(\Gamma_{\star\star}^{22})^{-1},...,(\Gamma_{\star\star}^{mm})^{-1}\right).$$
	\end{theorem} 
	
	\begin{remark}	
		A possible reasonable choice of the sequences of adaptive amounts is the following one (see \cite{zou})
\begin{equation}\label{eq:tuningzou}
\lambda_{n,j}^i=\frac{\alpha_n^i}{|\tilde\theta_{n,j}^i|^{\delta_i}},\quad i=1,2,...,m,
\end{equation}
where $(\alpha_n^i)_{n\geq 1}$ represents a sequence of positive real numbers satisfying the following conditions
\begin{equation}\label{eq:condtun}
r_n^i \alpha_n^i\longrightarrow 0,\quad (r_n^i)^{2-q_i-\delta_i}\alpha_n^i\longrightarrow\infty,
\end{equation}
with $\delta_i>1-q_i.$ Under the conditions \eqref{eq:condtun}, the assumptions B1-B3 fulfill.
\end{remark}
	\begin{remark}
	It is worth to mention that the estimator $\hat\theta_n$ is oracle when  the dimension $\mathfrak p$ and the the sparsity dimension $\mathfrak p^0$ are finite and fixed. We are not considering the high-dimensional  setting; i.e. $\mathfrak p\to \infty$ (and simultaneously $\mathfrak p^0\to \infty$ ) as $n\to\infty.$ Penalized statistics when number of parameters diverges has been studied, for instance, in \cite{peng}. The study of the Bridge-type estimator \eqref{eq:LASSOtypest} in the high-dimensional case represents a future research topic.
	\end{remark}

\section{Application to Stochastic Differential Equations}\label{sec4}
	\subsection{Ergodic diffusions} Let  $(\Omega,\cor F,{\bf F}=(\cor F_t)_{t\geq 0},P)$ be a filtered complete probability space.
	Let us consider a $d$-dimensional solution process $X:=(X_t)_{t\geq 0}$ to  the following stochastic differential equation (SDE)
	\begin{equation}\label{sde}
	\de X_t=b(X_{t},\alpha)\de t+\sigma(X_{t},\beta)\de W_t,\quad X_0=x_0,
	\end{equation}
	where $x_0$ is a deterministic initial point, $b:\mathbb{R}^d\times  \Theta_{\alpha}\to \mathbb{R}^d$ and $\sigma:\mathbb{R}^d \times \Theta_\beta\to \mathbb{R}^d\otimes \re^r$ are Borel known functions (up to $\alpha$ and $\beta$) and  $(W_t)_{t\geq 0}$ is a $r$-dimensional standard $\cor F_t$-Brownian motion. Furthermore, $\alpha\in\Theta_\alpha\subset \re^{p_1},\beta\in\Theta_\beta\subset \re^{p_2},p_1, p_2\in\mathbb N,$ are unknown parameters where $\Theta_\alpha,\Theta_\beta$ are compact convex sets. Let $\theta:=(\alpha,\beta)\in\Theta:=\Theta_\alpha\times\Theta_\beta$ and denote by $\theta_0:=(\alpha_0,\beta_0)$ the true value of $\theta$. Let us assume that $\theta_0\in$ Int$(\Theta)$ and $0\in\mathbb R^{p_1+p_2}$ belongs to $\Theta.$ The stochastic differential equation $X$ represents a sparse parametric model; that is $\theta_0$ has a sparse representation.

	The sample path of $X$ is observed only at $n+1$ equidistant discrete times $t_i^n$, such that $t_i^n-t_{i-1}^n=\Delta_n<\infty$ for $i=1,...,n,$ (with $t_0^n=0$). Therefore the data are the discrete observations of the sample path of $X,$ that we represent by ${\bf X}_n:=(X_{t_i^n})_{0\leq i\leq n}.$ Let $p$ an integer with $p\geq 2,$ the asymptotic scheme adopted in this paper is the following: $n\Delta_n\longrightarrow \infty$, $\Delta_n\longrightarrow 0$ and $n\Delta_n^p\longrightarrow 0$ as $n\to \infty$ and there exists $\epsilon \in(0,(p-1)/p)$ such that $n^\epsilon\leq n\Delta_n$ for large $n$.
	
	$X$ satisfies some mild regularity conditions (see, e.g., \cite{kess} or
\cite{yospol}). For instance, the functions $b$ and $\sigma$ are smooth, $\Sigma(x,\beta):=\sigma \sigma'(x,\beta)$ is supposed invertible and $X$ is an ergodic diffusion.; i.e. there exists a unique invariant probability measure $\mu = \mu_{\theta_0}$ such that for any bounded measurable function $g : \mathbb R^d\to  \mathbb R$, the convergence 
 $\frac1T\int_0^T g(X_t)\de t\stackrel{p}{\longrightarrow}\int_{\mathbb R^d}g(x)\mu(\de x).$
 
 We are interested to the estimation of $\theta$ as well as the correct identification of the zero coefficients by using the data ${\bf X}_n.$ For this reason, we apply the Bridge-type estimator \eqref{eq:LASSOtypest} in this setting. We assume that an initial estimator $\tilde\theta_n:=(\tilde\alpha_n,\tilde\beta_n)':\mathbb R^{(n+1)\times d}\longrightarrow \overline \Theta$ of $\theta $ satisfies  the following asymptotic properties:

\begin{itemize}
\item[(i)] $\tilde\alpha_n$ is $\sqrt{n\Delta_n}$-consistent while $\tilde\beta_n$ is $\sqrt{n}$-consistent; i.e. $(\sqrt{n\Delta_n}(\tilde\alpha_n-\alpha_0),\sqrt n(\tilde\beta_n-\beta_0))'=O_p(1);$
 \item[(ii)]   $\tilde\theta_n$ is asymptotically normal; i.e
 $$(\sqrt{n\Delta_n}(\tilde\alpha_n-\alpha_0),\sqrt n(\tilde\beta_n-\beta_0))'\stackrel{d}{\longrightarrow}N_{p_1+p_2}(0,\text{diag}((\Gamma^{11})^{-1},(\Gamma^{22})^{-1})),$$
where
$$ \Gamma^{11}:=\int_{\re^d}\left(\partial_\alpha b(\alpha_0,x)\right)' \Sigma^{-1}(\beta_0,x)\partial_\alpha b(\alpha_0,x)
\mu(\de x)\,,$$
$$ \Gamma^{22}:=\frac12\int_{\re^d}\text{tr}\left[(\partial_\beta \Sigma)\Sigma^{-1}(\partial_\beta \Sigma)\Sigma^{-1}(\beta_0,x)\right]
\mu(\de x),$$
where $\partial_\alpha:=(\frac{\partial}{\partial_{\alpha_1}},...,\frac{\partial}{\partial_{\alpha_{p_1}}})', \partial_\beta:=(\frac{\partial}{\partial_{\beta_1}},...,\frac{\partial}{\partial_{\beta_{p_2}}})'.$ We assume the integrability and the non-degeneracy of $\Gamma^{11}$ and $\Gamma^{22}.$
\end{itemize}
Therefore, from (i) and (ii) emerge that the estimator $\tilde\theta_n$ works in a mixed-rates asymptotic regime with two different rates, $\sqrt{n\Delta_n}$ and $\sqrt n,$ for the two groups of parameters $\alpha$ and $\beta.$ The assumptions A2 and A3 hold by setting $A_n= \text{diag}(1/(\sqrt{n\Delta_n}){\bf I}_{p_1},1/\sqrt n{\bf I}_{p_2})$.

Let $q_1,q_2\in(0,1],$ the objective function \eqref{eq:LASSOtype} becomes
\begin{equation}
\mathcal{F}_n(\theta):=(\theta-\tilde{\theta}_n)'\hat G_n(\theta-\tilde{\theta}_n)+\sum_{j=1}^{p_1}\lambda_{n,j}|\alpha_j|^{q_1}+\sum_{k=1}^{p_2}\gamma_{n,k}|\beta_k|^{q_2},
\end{equation}
where $\hat G_n$ is a $(p_1+p_2)\times (p_1+p_2)$ matrix assumed to be symmetric and positive definite and such that $A_n\hat G_nA_n\stackrel{p}{\longrightarrow}\text{diag}(\Gamma^{11},\Gamma^{22})$ (condition C1 fulfills). In this framework, we consider sequences $ \lambda_{n,j} $ and $ \gamma_{n,k} $ as in \eqref{eq:tuningzou}; i.e.  they turn out  as follows
	\begin{align*}
	\lambda_{n,j} = \frac{\lambda_{n,0}}{|\tilde \alpha_{n,j}|^{\delta_1}},
	\qquad 
	\gamma_{n,k} =  \frac{\gamma_{n,0}}{|\tilde \beta_{n,k}|^{\delta_2}}
	\qquad j = 1, , \ldots, p_1, \,\,k = 1, \ldots, p_2,
	\end{align*}
	where the exponents $ \delta_1 $ and $ \delta_2 $ are such that $ \delta_i > 1-q_i,1,2.$ We assume that
	$ \lambda_{n,0} $ and $ \gamma_{n,0} $ are deterministic sequences satisfying the conditions \eqref{eq:condtun}; that is
\begin{equation*}\frac{\lambda_{n,0}}{\sqrt{n\Delta_n}}\longrightarrow0,\quad (n\Delta_n)^{\frac{\delta_1-2+q_1}{2}}\lambda_{n,0}\longrightarrow\infty,\end{equation*}
and
\begin{equation*}\frac{\gamma_{n,0}}{\sqrt{n}}\longrightarrow0,\quad n^{\frac{\delta_2-2+q_2}{2}}\gamma_{n,0}\longrightarrow\infty,\end{equation*}
as $n\to\infty.$ Finally, the Bridge-type estimator for the stochastic differential equation \eqref{sde} becomes
\begin{equation}\label{eq:estsde}
\hat\theta_n:=(\hat\alpha_n,\hat\beta_n)'\in\arg\min_{\theta\in\overline \Theta} \mathcal F_n(\theta).
\end{equation}

In literature appeared different estimators for ergodic diffusions satisfying the asymptotic properties (i) and (ii). For instance, the quasi-maximum-likelihood estimator, the quasi-Bayesian estimator (see, \cite{yosjma}, \cite{kess}, \cite{yu}, \cite{yospol}, \cite{uchad}, \cite{uchbay}) and the hybrid multistep estimator (see \cite{kam}).

For $p=2,$ a suitable loss function could be the negative quasi-log-likelihood function \begin{align}\label{qlik}
\ell_n({\bf X}_n,\theta)
&:=\frac12\sum_{i=1}^n\Bigg\{\log\text{det}(\Sigma( X_{t_{i}^n},\beta))
\notag\\
&\quad+\frac{1}{\Delta_n}(X_{t_i^n}-X_{t_{i-1}^n}-\Delta_n b( X_{t_{i}^n},\alpha))'\Sigma^{-1}( X_{t_{i}^n},\beta)(X_{t_i^n}-X_{t_{i-1}^n}-\Delta_n b( X_{t_{i}^n},\alpha))\Bigg\}.
\end{align}
Therefore, a possible choice of $\hat G_n$ is the Hessian matrix $\ddot \ell_n({\bf X}_n,\theta)$ and $$\tilde\theta_n^{QL}\in \arg\min_{\overline \Theta} \ell_n({\bf X}_n,\theta)$$ represents the quasi-maximum likelihood estimator (QMLE). For more details on the quasi-likelihood analysis for stochastic differential equations, the reader can consult, for instance, \cite{kess} and \cite{yospol}.

Now, we are able to present the oracle properties for the Bridge-type estimator introduced in the framework of sparse ergodic diffusions. We also argue about the boundedness of the estimator which is useful for the moment convergence. Under the assumptions of Section \ref{sec3}, the next theorem concerns these issues in the case of initial estimator equals to $\tilde\theta_n^{QL}.$ Clearly, the statement of the theorem holds true also if $\tilde\theta_n$ is the quasi-Bayesian estimator or the hybrid multistep estimator.

\begin{theorem}\label{th:sde}
If $\tilde\theta_n=\tilde\theta_n^{QL},$ the Bridge-type estimator \eqref{eq:estsde} has the following properties:
\begin{itemize}
\item (Consistency) $ (\sqrt{n\Delta_n}(\hat\alpha_n-\alpha_0),\sqrt n(\hat\beta_n-\beta_0))=O_p(1);$
\item(Selection consistency) $P(\hat\alpha_{n\bullet}=0)\longrightarrow 1$ and $P(\hat\beta_{n\bullet}=0)\longrightarrow 1$;
\item (Asymptotic normality) $$(\sqrt{n\Delta_n}(\tilde\alpha_n-\alpha_0)_\star,\sqrt n(\tilde\beta_n-\beta_0)_\star)\stackrel{d}{\longrightarrow}N_{p_1^0+p_2^0}(0,\text{diag}((\Gamma_{\star\star}^{11})^{-1},(\Gamma_{\star\star}^{22})^{-1}));$$
\item (Uniformly $L^q$-boundedness) if $\sup_n \mathbb E[||\hat{\mathfrak D}_n||^q]<\infty$  and $\sup_n \mathbb E[||\hat{\mathfrak D}_n^{-1}||^q]<\infty$ for all $q\geq 1,$ we have that
$$\sup_n \mathbb E[|A_n^{-1}(\hat\theta_n-\theta_0)|^q]<\infty.$$
\end{itemize}
\end{theorem}
Theorem \ref{th:sde} represents a generalization of the results previously obtained in \cite{deg}. 

For the identiification of the true model, it is also possible to use AIC criterion as discussed in \cite{uch}). Nevertheless, as pointed out in \cite{iacus}, it is necessary to specify some parametric models.

\subsection{Tuning parameter selection}

	Consider a diffusion process as in \eqref{sde}. 
 In the linear regression problem, the regularized estimates are obtained by choosing the tuning parameters  by means of a cross-validation procedure. In our framework, we propose a data driven technique for choosing the tuning parameters for the penalized estimation problem \eqref{eq:estsde}. 
		
	Consider the Euler discretization of the solution of \eqref{sde}
	\begin{equation}\label{eq:sde-euler}
		X_{t_{i+1}^n} = X_{t_i^n} + b(X_{t_i^n}, \alpha) \Delta_n +
		\sigma(X_{t_i^n}, \beta)  (W_{t_{i+1}^n}-W_{t_{i}^n})	
	\end{equation}
	where $ t_i $ and $ \Delta_n $ are specified as above. The ``standardized residuals''
	are then defined as
	\begin{equation}\label{eq:residuals}
		r_{t_i^n} =
		\Delta_n^{- 1/2}\Sigma^{-1/2}(X_{t_i^n}, \beta) (X_{t_{i+1}^n} - X_{t_i^n} - \Delta_n b(X_{t_i^n}, \alpha) ) \qquad i = 1,\ldots, n.
	\end{equation}
	The residuals $ r_{t_i^n} $ are $  N_d(0_d,{\bf I}_d) $
	and conditionally independent. The idea is to find the tuning parameters in such a way that the residuals fit best to a white noise scheme.
	
	Given a series of observations $ {\bf X}_n $ and some estimates of the parameters 
	$ \hat \alpha $ and $ \hat \beta $ the residuals can be estimated as 
	\begin{equation}\label{eq:resid-est}
	\hat r_{t_i^n} = 
		\Delta_n^{- 1/2}
	\Sigma^{-1/2}(X_{t_i^n}, \hat \beta) (X_{t_{i+1}^n} - X_{t_i^n} - \Delta_n b(X_{t_i^n}, \hat \alpha) ) \qquad i = 1,\ldots, n
	\end{equation}
	Let $ \psi := (q_1, q_2, \lambda_n, \gamma_n, \delta_1, \delta_2) $ be the vector of 
	tuning parameters varying in some suitable parameter space $ \Psi \subset \mathbb R^6 $.
	Beside the penalized estimates will depend on $ \psi $ and consequently also the 
	residuals will. We set $\hat r_{t_i^n}  = \hat r_{t_i^n}(\psi)$ to stress this fact. 
	We can choose a desirable value for $ \psi $ by optimizing some score function 
	which penalizes tuning parameters producing residuals which deviate most from the hypothesis of being uncorrelated.
	More formally let $ S : \mathbb R^{nd} \mapsto\mathbb R^+ $ be such a score function
	which takes in input the $ d$-dimensional residuals and returns a low score  
	if the residuals appear to be incorrelated and a high value otherwise (low score is better).
	We choose the optimal value of the tuning parameter vector $ \psi^* $ as
	\begin{equation}\label{eq:psi-star}
	 \psi^* = \arg \min_{\psi \in \Psi} S(r_{t_1^n}(\psi), \ldots, r_{t_n^n}(\psi) )
	\end{equation}
	The penalty function can be the test statistic in a white noise hypothesis testing. In the numerical simulations we consider the Ljung-Box test statistic defined as 
	\begin{equation}\label{eq:ljung}
	Q_\ell(r_n) = n(n+2)\sum_{j=1}^{\ell} \frac{\hat \rho ^2_j(r_n)}{n - j}
	\end{equation}
	where $ r_n = (r_{t^n_i})_{ i=0}^n $ is a vector of residuals, $\ell$ is the number of lags to be tested, $ \hat \rho_j $ denotes the sample auto-correlations of the residuals
	at lag $ j $ 
	\begin{equation}
	\hat \rho ^2_j(r_n) = 
	\frac{ 
		\frac{1}{n - j}\sum_{i=1}^{n-j} (r_{t^n_i} - \bar r_n)(r_{t^n_{i+j}} - \bar r_n)
	}{
		\frac{1}{n }\sum_{i=1}^{n} (r_{t^n_i} - \bar r_n)^2,
	}
	\end{equation}
	where $ \bar r_n = n^{-1} \sum_{i=1}^n r_{t^n_i} $ and $ n $ is the number of observations. Under the hypothesis that the residuals are 
	not correlated up to lag $ \ell $, $ Q_\ell $ is asymptotically distributed as a $ \chi^2_\ell $. 
	
	A similar approach can be adapted if one wants to tune the tuning parameter in order to optimize the fit of the residuals to the Gaussian distribution. This idea was introduced in \cite{bandi} in the context of bandwidth selection for non-parametric estimates of the drift and diffusion coefficients. One can consider a penalty
	  measuring the distance of the empirical distribution of the residuals
	from the Gaussian distribution function (such as the Kolmogorov -- Smirnov test statistic).
	More formally, equip the space of distribution functions with some norm $ \| \cdot \|$. The parameter $ \psi^* $ can then be chosen as 
	\begin{equation}\label{eq:psi-star-gaus}
			\psi^* = \arg \min_{\psi \in \Psi} \| \hat F_n(r_{t_1^n}(\psi), \ldots, r_{t_n^n}(\psi) ) - P_d \|,
	\end{equation}
	 where $ \hat F_n $ denotes the empirical distribution function of the residuals and $ P_d $ denotes 
	 the distribution function of the $ d- $ dimensional standard Gaussian distribution.
	 In the case where the sup norm is chosen one recovers the Kolmogorov-Smirnov test statistic. 
	 
	 Two ore more criteria can be combined in order in such a way to minimize simultaneously over 
	 multiple score functions. 
	 In the following we consider together the criteria based on \eqref{eq:ljung} 
	 and \eqref{eq:psi-star-gaus} in order to seek residuals 
	 which are uncorrelated and Gaussian, in the sense that they 
	 minimize both the two scores. The tuning parameters are then computed as
	 the solution of the following optimization problem
	 \begin{equation}\label{eq:psi-star-comb}
	 	\psi^* = \arg \min_{\psi \in \Psi} 
	 	[ Q_\ell(r_n(\psi)) + \|\hat{F}_n(r_n(\psi)) - P_d\| ].
	 \end{equation}

	 \subsection{Algorithmic implementation} 
	 The following algorithm implements criterion \eqref{eq:psi-star-comb}.
	 \begin{itemize}
	 	\item Step 0. Suppose a set of data points $ {\bf X}_n$ is given. Initialize
	 	the tuning parameter vector $ \psi $ with some value $ \psi_0 $. Fix a threshold
	 	$ \epsilon >0 $. 
	 	\item Until convergence is reached:
	 	\begin{itemize}
	 		\item Step 1. Compute the current Bridge estimates 
	 		with the current value $ \psi^{(k)} $ of the tuning parameters
	 		$ \hat \alpha ^{ (k)} = \hat \alpha (\psi^{(k)})$, 
	 		$ \hat \beta ^{ (k)} = \hat \beta (\psi^{(k)})$.
	 		\item Step 2. Compute the residuals $ (\hat r_{t_i^n} ^{(k)} )_{i=0}^n:= (\hat r_{t_i^n} (\psi^{(k)}))_{i=0}^n$ as in fomula \eqref{eq:resid-est}, with the current
	 		estimates of the parameters $ \hat \alpha ^{ (k)}$ and $ \hat \beta ^{ (k)}$.
	 		\item Step 3. Evaluate the score of the current residuals
	 		$ s^{(k)} = S(\hat r_{t_1^n} ^{(k)}, \ldots , \hat r_{t_n^n} ^{(k)}) $
	 		\item Step 4.
	 		If $ |s^{(k)} - s^{(k-1)}| < \epsilon $ stop: convergence is reached. Set $ \psi^* = \psi^{(k)} $ and return the optimal Bridge estimates of the parameters 
	 		$ \alpha^* = \alpha^{(k)} $ and $ \beta^* = \beta^{(k)} $.
	 		Otherwise move to some new point $ \psi^{(k+1)} $ (chosen according to some optimization algorithm) and repeat Steps 1 to 4.
	 	\end{itemize}
	 \end{itemize}

\section{Numerical Results}\label{numres}
	 
\subsection{Simulation study}

Consider a multivariate diffusion process $ X = (X_t)_{ t\geq 0} $ driven by the SDE

\begin{align}\label{eq:sim-sde}
\begin{cases}
\mathrm d X^{(1)}_t = (\alpha_{10} +  \alpha_{11}X^{(1)}_t  + \alpha_{12} X^{(2)}_t  + \alpha_{13} X^{(3)}_t ) \mathrm d t + 
(\beta_{10} +   \beta_{11} X^{(1)}_t + \beta_{12} X^{(2)}_t +  \beta_{13} X^{(3)}_t) \mathrm d W^{(1)}_t\\
\mathrm d X^{(2)}_t  = (\alpha_{20} +  \alpha_{21}X^{(1)}_t  + \alpha_{22} X^{(2)}_t  + \alpha_{23} X^{(3)}_t ) \mathrm d t + 
(\beta_{20} +   \beta_{21} X^{(1)}_t + \beta_{22} X^{(2)}_t +  \beta_{23} X^{(3)}_t) \mathrm d W^{(2)}_t \\
\mathrm d X^{(3)}_t = (\alpha_{30} +  \alpha_{31}X^{(1)}_t  + \alpha_{32} X^{(2)}_t  + \alpha_{33} X^{(3)}_t ) \mathrm d t + 
(\beta_{30} +   \beta_{31} X^{(1)}_t + \beta_{32} X^{(2)}_t +  \beta_{33} X^{(3)}_t) \mathrm d W^{(3)}_t
\end{cases}
\end{align}
which can be written in compact matrix notation as
\begin{equation}\label{eq:sim-sde-vec}
\mathrm d X_t  = (\alpha_{ 0} +
A X_t) \mathrm d t + 
\text{diag}(\beta_{ 0} +
B X_t) \,  \mathrm d W_t
\end{equation}
\noindent
where $ X_t = (X_t^{(1)}, X_t^{(2)}, X_t^{(3)} ) $, $ W_t = (W_t^{(1)}, W_t^{(2)}, W_t^{(3)} ),$ with $ W_t^{(i)},i=1,2,3, $ independent Brownian motions, 
$ \alpha_{ 0} = (\alpha_{10},\alpha_{20},\alpha_{30})' $, 
$  \beta_{ 0} = (\beta_{10},\beta_{20},\beta_{30})' $,
$ A = (\alpha_{ij})_{1\leq i,j\leq 3} $, $ B = (\beta_{ij})_{1\leq i,j\leq 3}$.
A sample path of the solution to \eqref{eq:sim-sde-vec} is represented in \autoref{fig:sample-path-sim}. 
\begin{figure}
	\begin{overpic}[width = .7\linewidth]{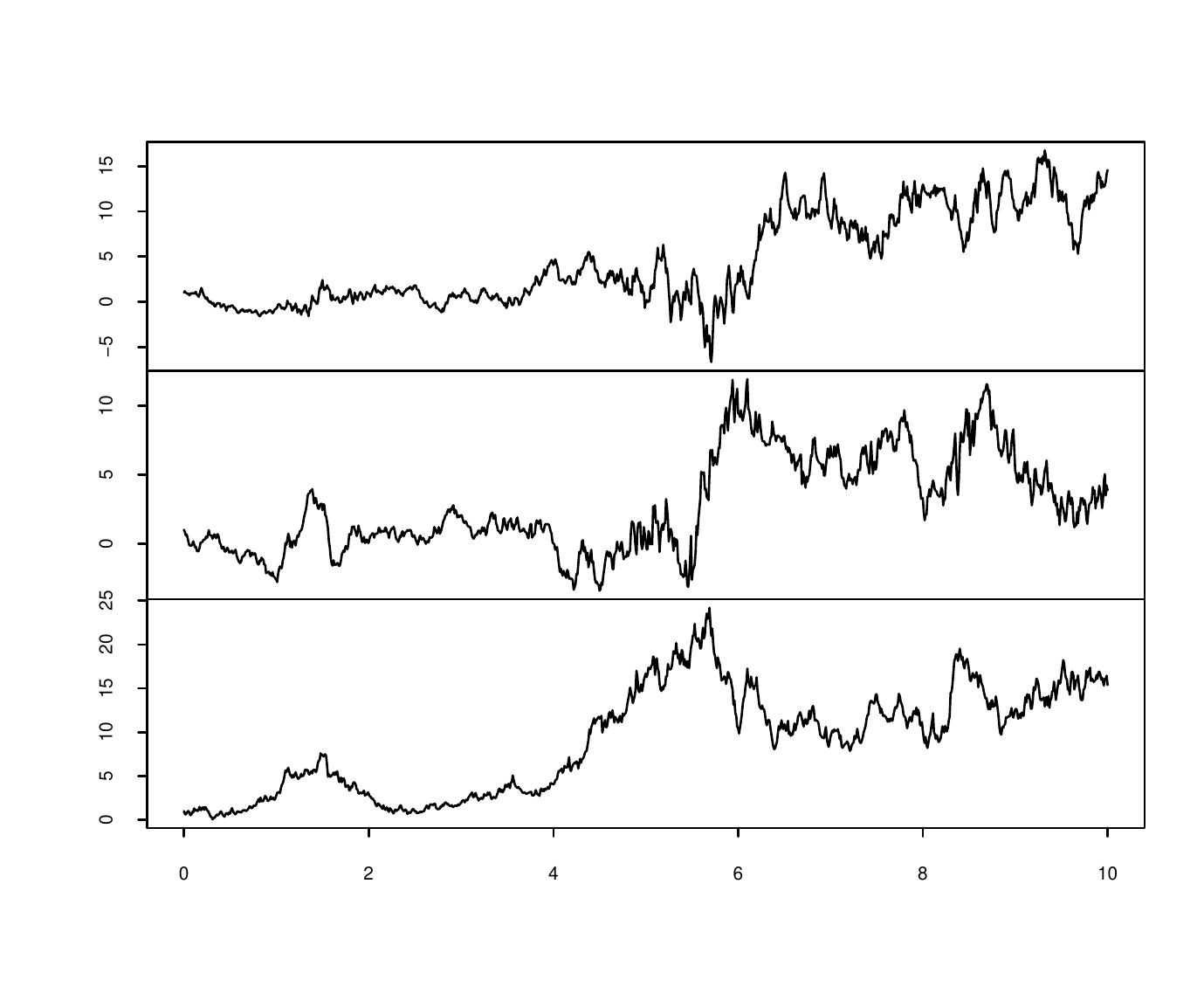}
		\put(55,5){	$t$}
		\put(0, 22){$X^{(3)}_t$}
		\put(0, 42){$X^{(2)}_t$}
		\put(0, 62){$X^{(1)}_t$}
	\end{overpic}
	\caption{A sample path of the solution to \eqref{eq:sim-sde-vec}}
	\label{fig:sample-path-sim}
\end{figure}
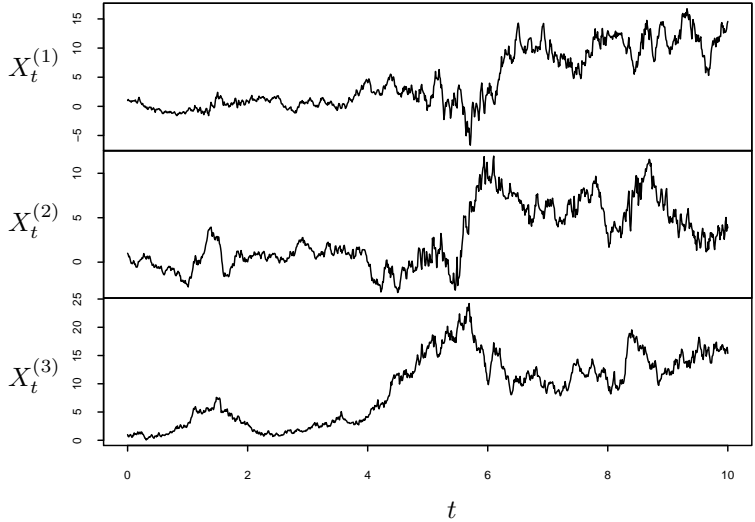

In our simulation we set several parameters to zero:
\[ \alpha_{21} = \alpha_{31} = \alpha_{32} = \alpha_{33} =   
\beta_{11} = \beta_{21} = \beta_{31} = \beta_{22} = \beta_{32} = 0. \] 
The true values for the non zero parameters are displayed in \autoref{tab:results}.
In particular this choice of the parameters implies that 
model \eqref{eq:sim-sde-vec} can be interpreted in terms 
of Granger causality. 
The idea is that the three components of $ X _t$ are correlated, but the value of $ X_t^{(2)}  $ effects $ X_t^{(1)} $, and $ X_t^{(3)} $ influences both $ X_t^{(1)} $ and $ X_t^{(2)} $ and not vice versa. Formally, in the 
continuous time setting we have the following definition of
non--causality (see \cite{mccr}). Suppose 
$ 	Z_t = (Y^{(1)}_t, Y^{(2)}_t, W_t)' $ is an 
$ n $--dimensional process, where $Y_t^{(1)},Y_t^{(2)} $ and $ W $ have dimension $  n_1, n_2$ and $ n_3 $, respectively, with $ n_1 + n_2 + n_3 = n $.
We say that $ Y_t^{(1)} $ does \emph{not} Granger--cause $ Y_t^{(2)} $ if
\begin{equation}\label{eq:granger-def}
\forall t, h \geq 0, \qquad 
\mathbb E(Y^{(2)}_{t+h}| \mathcal I_t) = \mathbb E (Y^{(2)}_{t+h}|\mathcal I_t - \mathcal Y^1_t),
\end{equation} 
where $ \mathcal I_t = \sigma( Z_s, s \leq t ) $ and
$ \mathcal I_t - \mathcal Y^1_t = \sigma( (Y^{(2)}_s, W_s), s \leq t )$.
Clearly, in the model \eqref{eq:sim-sde}, $(X_t ^{(1)}, X_t ^{(2)}) $ does not 
Granger-cause $ X_t ^{(3)} $, by setting 
$ Y_t ^{(1)} = (X_t ^{(1)}, X_t ^{(2)}) $ and $  Y_t ^{(2)} = X_t ^{(3)} $ in 
definition \eqref{eq:granger-def},
but the contrary is not true. Analogously $ X_t ^{(1)} $
does not Granger-cause $ X_t ^{(2)} $. 

\begin{remark}
 It is worth observing that the model \eqref{eq:sim-sde-vec} may not satisfy some assumption such a, for instance, the ergodicity (see Remark 1 in \cite{uchad}, for a sufficient condition). Nevertheless, it would be possible to modify slightly the SDE \eqref{eq:sim-sde-vec}, in order to guarantee that the assumptions fulfill. For instance, each diffusion term in the equations \eqref{eq:sim-sde} could be replaced with the following function
 $$\sigma_i(x,\beta)=\begin{cases}
 \beta_{i0}+\sum_{j=1}^3
\beta_{ij}x_j,& | \beta_{i0}+\sum_{j=1}^3
\beta_{ij}x_j |<M,\\
M,& \text{otherwise},
\end{cases}\quad i=1,2,3,$$
where $M>0$ is a positive constant sufficiently large. Therefore $$\Sigma(x,\beta)=\text{diag}(\sigma_1^2(x,\beta),\sigma_2^2(x,\beta),\sigma_3^2(x,\beta))$$ turns out bounded, which is a required condition  for the ergodicity.
\end{remark}

The aim of this simulation is the ability to recover the true 
model from the full model which contains a number
of unnecessary relations between the variables.
Moreover, we want to verify that the  multiple--penalties Bridge estimation technique,
together with the tuning parameter calibration procedure described above,
is able to identify the relevant relations among many. As a benchmark comparison we 
juxtapose the results of the Bridge estimation with those of the LASSO method and, 
with the un-penalized QMLE.  In order to compute the Bridge estimates, we will use as initial estimator $\tilde\theta_n^{QL}$ (with $p=2$) and $\hat G_n=\ddot{\ell}({\bf X}_n,\theta).$

We simulated $ N =   10^3 $ trajectories from 
model \eqref{eq:sim-sde-vec} over a long time interval and a fine grid, 
according to a  high frequency sampling scheme by setting
$ \Delta_n = n^{-1/3}$. We tested our model with increasing sample sizes equal to $ n = 500, \, n = 1000$ and $n = 10000$, in order to approach the asymptotic regime. The simulation was carried out 
in the context of the YUIMA framework (see the book \cite{iacus}), which provides the tools
for simulating sample paths of SDEs, performing quasi-maximum likelihood and adaptive LASSO estimation. 

Let $ \hat \alpha^{(k)}_{n,ij} $ and $ \hat \beta^{(k)}_{n,ij} $ denote the estimate obtained at replication $ k $
for the drift parameter $\alpha_{ij}  $ and 
for the diffusion parameter $ \beta_{ij} $, whose true values are denoted by $ \alpha_{0,ij}$ and $ \beta_{0,ij} $, respectively.
The performance of each estimation technique is evaluated 
by computing the empirical mean square errors
\begin{align}\label{eq:mse}
	\widehat{MSE}_{1ij} &= \frac 1 N \sum_{k=1}^N (\hat \alpha^{(k)}_{n,ij} - \alpha_{0,ij})^2  \qquad 1 \leq i \leq 3, \, 0 \leq j \leq 3\\
	\widehat{MSE}_{2ij} &= \frac 1 N \sum_{k=1}^N (\hat \beta^{(k)}_{n,ij} - \beta_{0,ij})^2   \qquad 1 \leq i \leq 3, \, 0 \leq j \leq 3 \notag 
\end{align}
and by means of the empirical selection frequencies, i.e. number of times that a null parameter was estimated as zero. 
The quantity \eqref{eq:mse} and the selection frequencies are computed for each of the estimation methods we want to compare. 
The numerical results are summarized in \autoref{tab:results} and \autoref{tab:sel-prop} respectively. 
 
With respect to  the experiment we conducted, we can draw the following conclusions, in particular we focus on model 
selection and thus identification of ``true" causal relations in the context of Granger causality. 

\begin{enumerate}
	\item
	\emph{The Bridge procedure can boost uniformity in model identification.}
	By looking at the selection frequencies (\autoref{tab:sel-prop}) we immediately note the different 
	behavior of the estimators for the drift and diffusion parameters: 
	each technique performs worse on the $\alpha_{ij}$'s. The 
	selection of the diffusion parameters is generally more accurate, almost reaching consistency
	for $n = 10000$. 
	This agrees with the results of \autoref{th:sde}: from a theoretical point 
	of view we expect a much slower convergence rate for the drift parameters. 
	We see that  
	the selection probability of the un-penalized estimator is generally lower
	then that of penalized techniques, as expected.
	The LASSO and the Bridge both show high selection probabilities for the diffusion parameters,
	with the LASSO slightly better. But the main difference lies in the behavior of the drift parameters estimators: the Bridge has a selection probability several times higher than the LASSO, especially 
	for moderate sample sizes.  
	That is the Bridge estimator has a generally higher accuracy in the selection of the 
	\emph{whole} model, losing a bit of accuracy on the diffusion parameters group 
	but gaining a much greater improvement for the drift parameters group.
	To illustrate this point we chose one representative of the drift parameters 
	and one of the diffusion parameters. In  \autoref{fig:emp-dist} we show the empirical distributions of the three types of estimators
	of the chosen representatives (i.e. the estimators of $\alpha_{21}$ and $\beta_{21}$), for each sample size. 
	The results of the empirical mean squared errors 
	and selection probabilities allow us to conclude that 
	the adaptive Bridge penalization 
	has a good performance uniformly with respect to the several parameter groups.
	
	\item \emph{The Bridge estimator is less sensible to a poor initial guess.}
	In the cases where the LASSO estimator gives its better results the  quasi-maximum likelihood
	estimator is  more concentrated around zero too. This means that the initial estimates 
	for the adaptive procedure
	were usually quite good and produced higher weights for the null components of 
	the vector parameters. The Bridge estimator on the other hand is able to identify 
	the zero parameters with a higher frequency also when the distribution 
	of the QMLE is more spread apart. This suggests that the performance of 
	the Bridge estimator is, at least in this example, less subordinate to a good 
	performance of the initial estimate with respect to the LASSO procedure. 
	
	\item 
	\emph{The automatic tuning parameter selection criterion is conservative. }
	The initial values for the tuning parameters were set to $ \lambda_{n,0} = \gamma_{n,0} = 2, \delta_1 = \delta_2 = 1, q_1 = q_2 = .9 $. It is worth noticing that 
	the automatic tuning parameter choice procedure does only a fine 
	adjustments around the initial value supplied. This means that 
	the user can choose the magnitude of penalization he desires
	and the algorithm will tweak it to better fit the data. 
	The results of the  tuning parameter selection procedure obtained in our simulation study, for $n = 1000$, are summarized 
	in \autoref{tab:tuning-par}. 
\end{enumerate}

{ \scriptsize
\begin{table}
		 \caption{Summary of the results of the simulation study. }\label{tab:results}
	%
	
	\begin{tabular}{c ccccccc } 
		\hline
		\multirow{2}{*}{Par.}
		&
		\multicolumn{2}{l}{Bridge } &
		\multicolumn{2}{l}{LASSO} &
		\multicolumn{2}{l}{QMLE} &
		\multirow{2}{*}{True value} \\
		& Avg.  (St. Err) & $ \widehat{MSE}$ &
		Avg.  (St. Err) & $ \widehat{MSE}$ &
		Avg.  (St. Err) & $ \widehat{MSE}$ & \\
		\hline
		\hline 
		$\alpha_{10}$  & -1.4576 (1.2738) & 1.6234 & -1.6417 (1.386) & 1.9397 & -1.6737 (1.3784) & 1.9289 & -1.5\\ 
		$\alpha_{11}$  & -0.9376 (0.8198) & 0.9879 & -2.0592 (0.9213) & 1.161 & -2.0692 (0.9044) & 1.1414 & -1.5\\ 
		$\alpha_{12}$  & 0.2629 (0.5437) & 0.5327 & 1.0277 (0.8829) & 0.8561 & 1.0459 (0.8748) & 0.8523 & 0.75\\ 
		$\alpha_{13}$  & 0.776 (0.6678) & 0.4463 & 1.0245 (0.7355) & 0.616 & 1.0431 (0.7283) & 0.616 & 0.75\\

		$\alpha_{20}$  & -0.3547 (0.6353) & 1.7152 & -1.7171 (1.3249) & 1.8013 & -1.7309 (1.3206) & 1.796 & -1.5\\ 
		$\alpha_{21}$  & 0.1177 (0.3624) & 0.1451 & 0.0286 (0.6292) & 0.3964 & 0.0229 (0.6321) & 0.3998 & 0\\ 
		$\alpha_{22}$  & -1.8366 (0.7923) & 0.7406 & -2.1356 (0.8421) & 1.1127 & -2.1412 (0.8399) & 1.1161 & -1.5\\ 
		$\alpha_{23}$  & 0.6343 (0.4934) & 0.2566 & 0.9938 (0.6388) & 0.4672 & 1.01 (0.64) & 0.4769 & 0.75\\

		$\alpha_{30}$  & 1.1661 (0.6206) & 0.4964 & 2.5298 (1.1192) & 2.3122 & 2.5465 (1.1169) & 2.3417 & 1.5\\ 
		$\alpha_{31}$  & 0.0007 (0.4084) & 0.1667 & -0.0315 (0.6012) & 0.3621 & -0.0333 (0.6062) & 0.3684 & 0\\ 
		$\alpha_{32}$  & -0.0064 (0.4386) & 0.1923 & -0.0286 (0.6702) & 0.4497 & -0.0269 (0.674) & 0.4546 & 0\\ 
		$\alpha_{33}$  & -0.0689 (0.3356) & 0.1173 & -0.2664 (0.5313) & 0.3531 & -0.2732 (0.5281) & 0.3534 & 0\\

		$\beta_{10}$  & 1.3649 (0.6112) & 0.3915 & 1.4397 (0.5957) & 0.3583 & 1.4935 (0.5859) & 0.3431 & 1.5\\ 
		$\beta_{11}$  & 0.0521 (0.3491) & 0.1245 & 0.0398 (0.339) & 0.1164 & 0.0561 (0.3749) & 0.1436 & 0\\ 
		$\beta_{12}$  & 0.6217 (0.6415) & 0.4604 & 0.6536 (0.7094) & 0.5672 & 0.7554 (0.6808) & 0.5894 & 0.4\\ 
		$\beta_{13}$  & 0.6152 (0.5373) & 0.3348 & 0.6254 (0.5772) & 0.3837 & 0.6761 (0.5781) & 0.4102 & 0.4\\

		$\beta_{20}$  & 1.3751 (0.4769) & 0.2429 & 1.457 (0.4639) & 0.2169 & 1.5149 (0.4549) & 0.2071 & 1.5\\ 
		$\beta_{21}$  & -0.003 (0.2767) & 0.0765 & -0.0136 (0.272) & 0.0741 & -0.0062 (0.3039) & 0.0923 & 0\\ 
		$\beta_{22}$  & 0.0118 (0.2698) & 0.0729 & 0.0039 (0.2566) & 0.0658 & 0.0143 (0.2927) & 0.0858 & 0\\ 
		$\beta_{23}$  & 0.5612 (0.4646) & 0.2417 & 0.5251 (0.5008) & 0.2663 & 0.5854 (0.5123) & 0.2966 & 0.4\\

		$\beta_{30}$  & 1.4165 (0.5219) & 0.2791 & 1.4205 (0.5028) & 0.259 & 1.4754 (0.4937) & 0.2442 & 1.5\\ 
		$\beta_{31}$  & 0.0554 (0.3063) & 0.0968 & 0.0467 (0.2939) & 0.0885 & 0.058 (0.3226) & 0.1074 & 0\\ 
		$\beta_{32}$  & 0.0155 (0.3137) & 0.0986 & 0.0146 (0.2896) & 0.084 & 0.018 (0.3241) & 0.1053 & 0\\ 
		$\beta_{33}$  & 0.5599 (0.4971) & 0.2725 & 0.5322 (0.5108) & 0.2783 & 0.5907 (0.5191) & 0.3056 & 0.4\\ \hline
		
		\hline
	\end{tabular}
	\subcaption{$n=500$}
	
	\begin{tabular}{c ccccccc } 
		\hline
		\multirow{2}{*}{Par.}
		&
		\multicolumn{2}{l}{Bridge } &
		\multicolumn{2}{l}{LASSO} &
		\multicolumn{2}{l}{QMLE} &
		\multirow{2}{*}{True value} \\
		& Avg.  (St. Err) & $ \widehat{MSE}$ &
		Avg.  (St. Err) & $ \widehat{MSE}$ &
		Avg.  (St. Err) & $ \widehat{MSE}$ & \\
		\hline
		\hline 
		$\alpha_{10}$  & -1.5139 (1.2947) & 1.676 & -1.6591 (1.3822) & 1.9353 & -1.6783 (1.372) & 1.9137 & -1.5\\		 
		$\alpha_{11}$  & -0.9119 (0.7743) & 0.9453 & -2.0197 (0.9178) & 1.1122 & -2.0419 (0.8857) & 1.0778 & -1.5\\ 
		$\alpha_{12}$  & 0.2774 (0.5604) & 0.5373 & 1.0386 (0.8749) & 0.8485 & 1.0519 (0.8651) & 0.8393 & 0.75\\ 
		$\alpha_{13}$  & 0.7671 (0.6585) & 0.4338 & 0.9992 (0.7405) & 0.6104 & 1.0133 (0.7236) & 0.5927 & 0.75\\

		$\alpha_{20}$  & -0.3403 (0.6324) & 1.7447 & -1.6824 (1.3291) & 1.7991 & -1.6921 (1.3239) & 1.789 & -1.5\\
		$\alpha_{21}$  & 0.1122 (0.3708) & 0.15 & 0.0288 (0.6019) & 0.3631 & 0.0296 (0.5981) & 0.3585 & 0\\ 
		$\alpha_{22}$  & -1.8588 (0.7865) & 0.7471 & -2.0806 (0.8562) & 1.07 & -2.0909 (0.8389) & 1.0526 & -1.5\\ 
		$\alpha_{23}$  & 0.6528 (0.4718) & 0.232 & 0.9831 (0.6509) & 0.4779 & 0.9988 (0.6416) & 0.4734 & 0.75\\ 
		
		$\alpha_{30}$  & 1.1897 (0.585) & 0.4384 & 2.5817 (1.0944) & 2.3672 & 2.6016 (1.0786) & 2.3765 & 1.5\\ 
		$\alpha_{31}$  & 0.007 (0.3917) & 0.1534 & -0.0177 (0.5583) & 0.312 & -0.0189 (0.5562) & 0.3096 & 0\\ 
		$\alpha_{32}$  & 0.0022 (0.4454) & 0.1984 & -0.0169 (0.6417) & 0.412 & -0.0135 (0.642) & 0.4122 & 0\\ 
		$\alpha_{33}$  & -0.0749 (0.3483) & 0.1269 & -0.2345 (0.5209) & 0.3262 & -0.2435 (0.512) & 0.3213 & 0\\ 
		
		$\beta_{10}$  & 1.3518 (0.674) & 0.4761 & 1.4854 (0.6191) & 0.3834 & 1.525 (0.6155) & 0.3793 & 1.5\\ 	
		$\beta_{11}$  & 0.0389 (0.3635) & 0.1336 & 0.0282 (0.3483) & 0.122 & 0.0395 (0.3771) & 0.1437 & 0\\ 		
		$\beta_{12}$  & 0.5937 (0.6152) & 0.4159 & 0.7112 (0.6889) & 0.5712 & 0.7741 (0.6874) & 0.6123 & 0.4\\		
		$\beta_{13}$  & 0.5835 (0.5317) & 0.3163 & 0.6615 (0.5822) & 0.4072 & 0.702 (0.593) & 0.4427 & 0.4\\

		$\beta_{20}$  & 1.3247 (0.5453) & 0.328 & 1.4639 (0.5058) & 0.2571 & 1.5015 (0.5093) & 0.2593 & 1.5\\ 
		$\beta_{21}$  & 0.0117 (0.311) & 0.0968 & -0.0022 (0.2984) & 0.089 & 0.0021 (0.3255) & 0.1059 & 0\\ 
		$\beta_{22}$  & 0.0137 (0.3313) & 0.1099 & 0.0119 (0.2975) & 0.0886 & 0.0253 (0.3294) & 0.1091 & 0\\ 
		$\beta_{23}$  & 0.5722 (0.5054) & 0.285 & 0.5925 (0.5377) & 0.3261 & 0.635 (0.5541) & 0.3622 & 0.4\\

		$\beta_{30}$  & 1.3948 (0.5406) & 0.3032 & 1.4599 (0.4897) & 0.2414 & 1.5036 (0.4846) & 0.2348 & 1.5\\
		$\beta_{31}$  & 0.0178 (0.2985) & 0.0894 & 0.0153 (0.2981) & 0.0891 & 0.0174 (0.3235) & 0.105 & 0\\ 
		$\beta_{32}$  & 0.004 (0.3178) & 0.101 & 0.0035 (0.2968) & 0.0881 & 0.0053 (0.3252) & 0.1057 & 0\\ 
		$\beta_{33}$  & 0.5406 (0.4914) & 0.2611 & 0.5723 (0.5241) & 0.3043 & 0.6183 (0.5464) & 0.3461 & 0.4\\

		\hline
	\end{tabular}
	\subcaption{$n=1000$}
	\bigskip


%
	%
	\begin{tabular}{c ccccccc } 

		\hline
		\multirow{2}{*}{Par.}
		&
		\multicolumn{2}{l}{Bridge } &
		\multicolumn{2}{l}{LASSO} &
		\multicolumn{2}{l}{QMLE} &
		\multirow{2}{*}{True value} \\
		& Avg.  (St. Err) & $ \widehat{MSE}$ &
		Avg.  (St. Err) & $ \widehat{MSE}$ &
		Avg.  (St. Err) & $ \widehat{MSE}$ & \\
		\hline
		\hline 
		$\alpha_{10}$  & -1.5743\newline (0.3311) & 0.1096 & -1.4107\newline (0.2665) & 0.0723 & -1.4919\newline (0.2817) & 0.0799 & -1.5\\ 
		$\alpha_{11}$  & -0.6301\newline (0.6105) & 1.1287 & -1.7535\newline (0.4647) & 0.2797 & -1.7576\newline (0.4825) & 0.2989 & -1.5\\ 
		$\alpha_{12}$  & 0.0909\newline (0.2492) & 0.4964 & 0.605\newline (0.3613) & 0.1513 & 0.6291\newline (0.3919) & 0.168 & 0.75\\ 
		$\alpha_{13}$  & 1.0865\newline (0.5839) & 0.4535 & 1.2614\newline (0.4958) & 0.5069 & 1.2625\newline (0.5173) & 0.5299 & 0.75\\ 
		$\alpha_{20}$  & -0.5559\newline (0.2808) & 0.1064 & -2.1026\newline (0.3989) & 0.1631 & -2.1682\newline (0.4172) & 0.1795 & -1.5\\ 
		$\alpha_{21}$  & 0.1623\newline (0.3033) & 0.1181 & 0.0336\newline (0.4542) & 0.207 & 0.0476\newline (0.459) & 0.2127 & 0\\ 
		$\alpha_{22}$  & -1.6534\newline (0.5847) & 0.3648 & -1.7702\newline (0.4627) & 0.2867 & -1.7595\newline (0.4739) & 0.2917 & -1.5\\ 
		$\alpha_{23}$  & 0.8658\newline (0.4269) & 0.1953 & 1.7437\newline (0.5375) & 1.2757 & 1.744\newline (0.5296) & 1.2682 & 0.75\\ 
		$\alpha_{30}$  & 1.5828\newline (0.3598) & 0.1294 & 1.7036\newline (0.2788) & 0.0776 & 1.7197\newline (0.2945) & 0.0867 & 1.5\\ 
		$\alpha_{31}$  & 0.0221\newline (0.2168) & 0.0474 & 0.0157\newline (0.2268) & 0.0516 & 0.0067\newline (0.2444) & 0.0597 & 0\\ 
		$\alpha_{32}$  & 0.0033\newline (0.2153) & 0.0463 & 0.004\newline (0.213) & 0.0453 & 0.0085\newline (0.2293) & 0.0526 & 0\\ 
		$\alpha_{33}$  & -0.1163\newline (0.2848) & 0.0945 & -0.1964\newline (0.3236) & 0.1431 & -0.1817\newline (0.3486) & 0.1544 & 0\\ 
		$\beta_{10}$  & 1.441\newline (0.268) & 0.0728 & 1.4335\newline (0.2691) & 0.0735 & 1.4754\newline (0.2701) & 0.0737 & 1.5\\ 
		$\beta_{11}$  & 0.0116\newline (0.0966) & 0.0094 & 0.0071\newline (0.0944) & 0.009 & 0.007\newline (0.1052) & 0.0111 & 0\\ 
		$\beta_{12}$  & 0.253\newline (0.1999) & 0.0615 & 0.3061\newline (0.171) & 0.038 & 0.3123\newline (0.1709) & 0.0369 & 0.4\\ 
		$\beta_{13}$  & 0.6054\newline (0.1768) & 0.0734 & 0.6014\newline (0.1801) & 0.0729 & 0.6131\newline (0.2053) & 0.0875 & 0.4\\ 
		$\beta_{20}$  & 1.8843\newline (0.2431) & 0.06 & 1.8948\newline (0.2583) & 0.0678 & 1.9126\newline (0.2682) & 0.0732 & 1.5\\ 
		$\beta_{21}$  & -0.007\newline (0.0702) & 0.005 & -0.0055\newline (0.073) & 0.0053 & -0.0057\newline (0.0768) & 0.0059 & 0\\ 
		$\beta_{22}$  & -0.0011\newline (0.0509) & 0.0026 & -0.0023\newline (0.0694) & 0.0048 & -0.0031\newline (0.069) & 0.0048 & 0\\ 
		$\beta_{23}$  & 0.8396\newline (0.1905) & 0.2295 & 0.8414\newline (0.1907) & 0.2312 & 0.8432\newline (0.1987) & 0.2358 & 0.4\\ 
		$\beta_{30}$  & 0.9928\newline (0.2341) & 0.0649 & 0.9846\newline (0.2369) & 0.0665 & 1.0082\newline (0.256) & 0.0752 & 1.5\\ 
		$\beta_{31}$  & -0.0054\newline (0.0512) & 0.0026 & -0.0065\newline (0.0651) & 0.0043 & -0.0062\newline (0.0647) & 0.0042 & 0\\ 
		$\beta_{32}$  & -0.0014\newline (0.062) & 0.0038 & 0.0002\newline (0.0574) & 0.0033 & -0.0023\newline (0.0721) & 0.0052 & 0\\ 
		$\beta_{33}$  & 0.4085\newline (0.1261) & 0.0159 & 0.4105\newline (0.1361) & 0.0186 & 0.4168\newline (0.1677) & 0.0284 & 0.4\\ 
		\hline
	\end{tabular}
	\subcaption{ $ n=10000$}

\end{table}
}

\bigskip \bigskip 
\begin{table}[h]
	\setlength{\extrarowheight}{.25em}

	\begin{tabular}{c|ccc|ccc|ccc}
			\hline 
		  \multicolumn{1}{c}{}
		& \multicolumn{3}{c}{$n=500$}
		& \multicolumn{3}{c}{$n=1000$}
		& \multicolumn{3}{c}{$n=10000$}\\
		\hline 
		Par. & Bridge & LASSO & QMLE &
		Bridge & LASSO & QMLE &
		Bridge & LASSO & QMLE \\
		\hline

$\alpha_{21}$  &   0.382  &   0.058   &  0.023 &	0.398  &  0.080 &  0.041 &	0.566 &    0.195  &   0.159 \\
$\alpha_{31}$  &   0.401  &   0.065   &  0.031 &	0.391  &  0.093 &  0.044 &	0.613 &    0.368  &   0.326 \\
$\alpha_{32}$  &   0.357  &   0.056   &  0.025 &	0.368  &  0.086 &  0.038 &	0.647 &    0.410  &   0.355 \\
$\alpha_{33}$  &   0.404  &   0.053   &  0.030 &	0.402  &  0.076 &  0.045 &	0.487 &    0.229  &   0.205 \\
$\beta_{11}$   &   0.501  &   0.676   &  0.169 &	0.605  &  0.691 &  0.324 &	0.948 &    0.967  &   0.962 \\
$\beta_{21}$   &   0.646  &   0.803   &  0.232 &	0.720  &  0.793 &  0.442 &	0.963 &    0.981  &   0.973 \\
$\beta_{22}$   &   0.588  &   0.772   &  0.172 &	0.685  &  0.778 &  0.359 &	0.981 &    0.985  &   0.976 \\
$\beta_{31}$   &   0.629  &   0.780   &  0.217 &	0.760  &  0.805 &  0.470 &	0.985 &    0.987  &   0.983 \\
$\beta_{32}$   &   0.610  &   0.771   &  0.178 &	0.713  &  0.791 &  0.392 &	0.983 &    0.987  &   0.983 \\

	\hline  
	\end{tabular}
\caption{Empirical selection frequencies for the zero parameters. }\label{tab:sel-prop}
\end{table}

\begin{figure}
	\centering
	\includegraphics[height=\textheight]{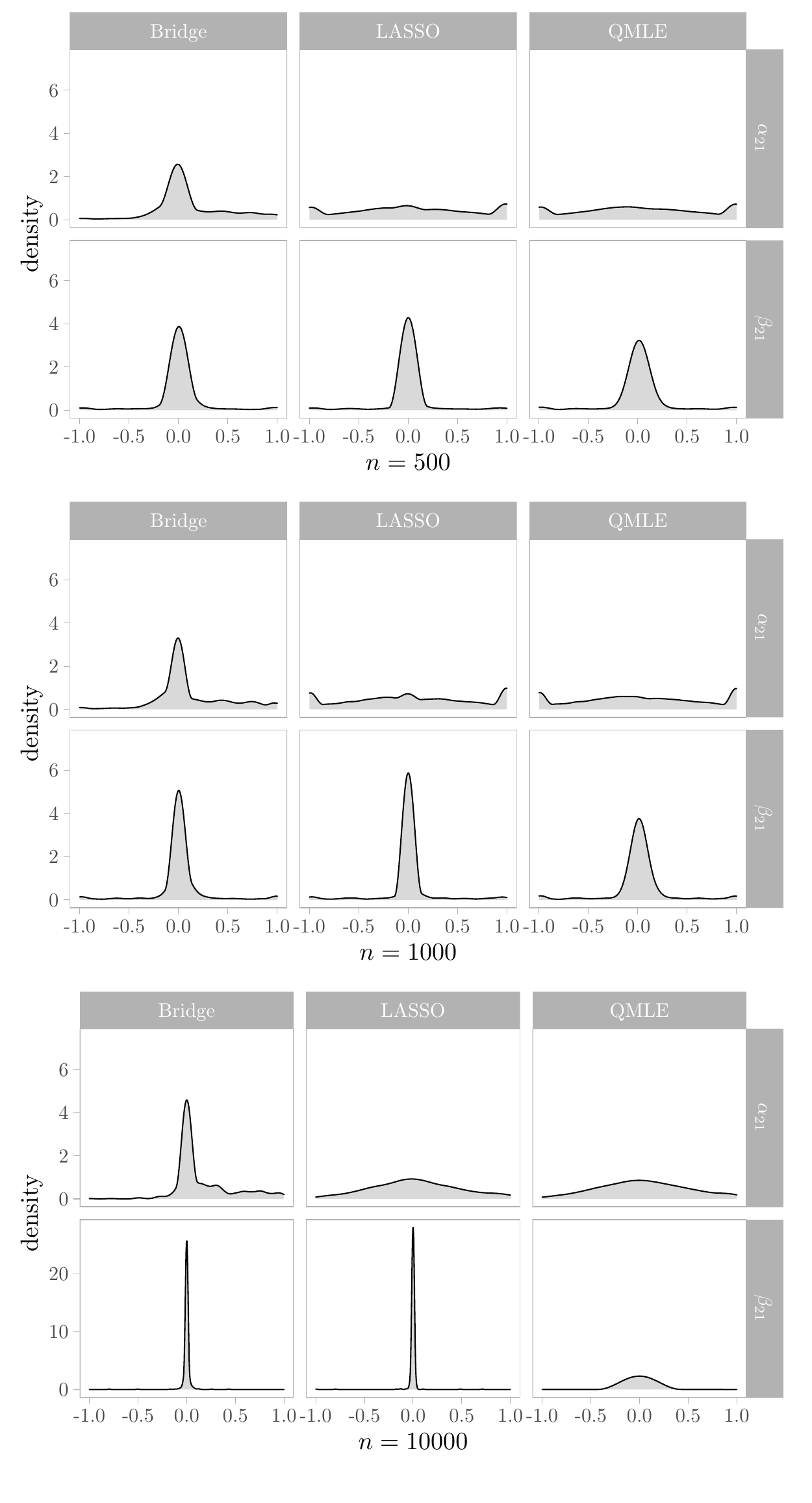}
	\caption{Comparison of the distributions of the estimators of two of the null parameters, $ \alpha_{21} $ and $ \beta_{21} $: Bridge (column 1), LASSO (column 2) and QMLE (column 3). Note that the plots on the last line are on a different scale.}
	\label{fig:emp-dist}
\end{figure}

\begin{table}
	\begin{tabular}{cccccccc}
		\hline
		& Min. & 1st Qu.  &Median  &  Mean & 3rd Qu.  &  Max. &  St. Dev. \\
		\hline \hline
		$ q_1 $ &
		0.001 &  0.89 &  0.94 &  0.94 &  0.98 &  1 & 0.05 \\
		$ q_2 $ &
		0.76 &  0.9 &  0.93 &  0.94 &  0.98 &  1 & 0.05 \\
		$ \lambda_{0} $ &
		0.1 & 1.69 &   1.99 &   2.06 &   2.27 &   7.37 & 0.93 \\
		$ \gamma_0 $ &
		0.1  & 1.56  & 1.98 &   1.96 &   2.04 &  10 & 0.87 \\ 
		$ \delta_1 $ &	0.1  &0.92 & 1 & 1.07 & 1.27 & 2  & 0.32 \\
		$ \delta_2 $ & 0.1 & 0.71 &  0.99&  0.92 & 1 & 2 & 0.34
		\\\hline
	\end{tabular}
	\caption{Summary of the tuning parameters obtained by implementing the adjustment procedure \eqref{eq:psi-star-comb} for $ n = 1000 $.}
	\label{tab:tuning-par}
\end{table}

\subsection{Real data prediction.}

One of the main purposes of the regularized procedures is to improve the predictive capability of the model. In fact, usually when we 
apply of statistical learning techniques we are not interested on the estimates of the parameters as much 
as we are in the ability of the model to provide accurate predictions for future outcomes. 
Bearing this in mind we tested our model on real data in a predictive study. We fitted a model of the form \eqref{eq:sim-sde-vec} with four components 
on a financial time series of daily closing stock prices of four major tech companies, Google, Amazon, Apple and Microsoft, which will be denoted 
by $ X_1, X_2, X_3, X_4$, respectively. 
The time series consists of $ n = 3283 $ observations starting Jan. 3rd, 2007. 
The goal is to predict the evolution of the price over a year long period. The training data 
consists of 3031 observations, until 16-01-2019. The test data is made of 
the last year of observations, from 17-01-2019 to 16-01-2020. The data have been downloaded 
by using the service Yahoo Finance and imported into R by using the library \texttt{quantmod}.
 
The scheme of the experiment is as follows. We first fitted the model \eqref{eq:sim-sde-vec}
on our training set by using the  adaptive Bridge and LASSO methods by using the QMLE as initial estimator. 
Then we performed parametric bootstrap to obtain simulations  of the series for the time period corresponding to the test data. 
In order to assess the performance of the estimators we computed predictive mean square errors and predictive confidence bands.
Let $ n_{te} $ be the number of observation in the test set $ (x_{t_i})_{i=1}^{n_{te}} $ and $ N $ the number of simulations performed. 
The corresponding predicted value are denoted by $ (\hat x^{(k)}_{t_i})_{ i=1}^{ n_{te}} , k=1, \ldots, N$. 
The predictive mean square error is computed as
\begin{equation}\label{key}
\widehat{MSE}_{p} = \frac{1}{n_{te} \cdot N} \sum_{i=1}^{n_{te}} \sum_{k=1}^{N} (x_{t_i} - \hat x^{(k)}_{t_i})^2
\end{equation}
The error bands are computed as the quantiles of the predicted values at each time instant. In this case we show 80\% and 95\% 
confidence bands. 

The results are summarized in \autoref{tab:pred-mses}. It compares 
the results obtained with the Bridge and LASSO technique  over $ N = 10^4 $ simulations for each of the stocks considered. The table also
reports the result obtained with the unpenalized QMLE as a benchmark. The tuning parameters have been set to $ \lambda_{0} = \gamma_0 = 10, \delta_1 = \delta_2 = 2.5 $ for both the LASSO and the Bridge estimator and  $ q_i,i=1,2,, $ was chosen to be $ 0.9 $.
We did not use any tuning parameter selection technique in this predictive
study. We adopted the rule to set to zero all the parameters 
for which the 
absolute value of  the 
estimate was below a certain threshold $ \epsilon $, thus obtaining a reduced model. We then
ran the bootstrap simulations with the reduced model for each technique.
Initially the full model had 40 parameters, 11 of which 
were estimated as zero by the Bridge estimator and 9 by the LASSO, 
having set $ \epsilon = 10^{-3} $.
\begin{table}
		\setlength{\extrarowheight}{.25em}
	\begin{tabular}{c c c c}
		\hline 
		& 
		\multicolumn{3}{c}{$ \widehat{MSE}_{p} $} \\
		 
		\textit{Series} & Bridge & LASSO & QMLE \\
		\hline \hline 
		$ X_1 $ & 4.64 & 9.47 &  10.78 \\
		$ X_2 $ & 2.75 & 5.96 &  6.09 \\
		$ X_3 $ & 4.68 & 9.34 &  10.52 \\
		$ X_4 $ & 8.12 & 11.4 &	 11.72 \\
		\hline 
	\end{tabular}
\caption{Comparison of predictive mean square errors over $ N = 10^4 $ simulations. }
\label{tab:pred-mses}
\end{table}

\begin{figure}
	\centering
	\includegraphics[width = \textwidth]{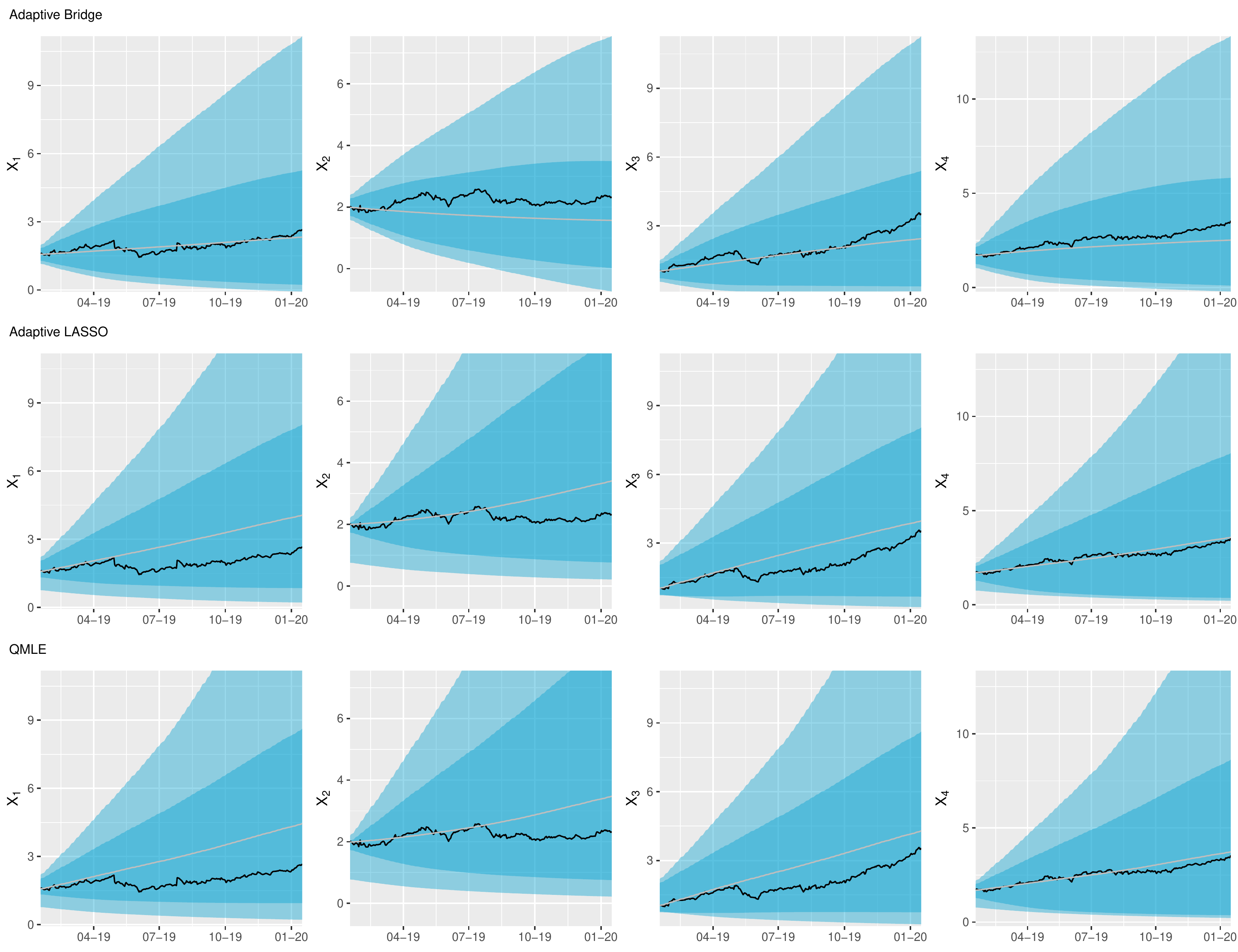}
	\caption{Predictive MSE obtained with the Bridge (row 1), LASSO (row 2) and QML estimators (row 3) for each of the four data series. The darker band shows depicts the 80\% quantiles, the lighter band the 95\% quantiles.}
	\label{fig:mse-pred}
\end{figure}

\begin{enumerate}
	\item \emph{Bridge estimator can achieve better predictive capability}.
	Results of \autoref{tab:pred-mses} show that the predictive MSE of the Bridge estimator is smaller on all the 
	three data series. The bridge estimator was able to produce improvements  on the predictive error of 
	57\%, 54\%, 55\% and 31\% for the three series with respect to the unpenalized estimator.  
	The LASSO improved the predictive capability  of the model as well, but in this case 
	the reduction was modest, 
	12\%, 2\% and 11\%, 2\% respectively, for the three series. 
	We arrive to the same conclusion by comparing the confidence bands depicted in \autoref{fig:mse-pred}. 
	The bands obtained with the Bridge estimator, on the first line, are narrower, leading to 
	less uncertainty in the prediction. 
	
	\item \emph{Bridge estimator resulted reliable over longer time periods}.
	By looking at \autoref{fig:mse-time} we see how the predictive mean square error changes over time. At first the two errors 
	are similar, but at later times the LASSO error grows at a higher rate then the error of the Bridge estimator, 
	reaching values some times even double in the last part of the trajectory. This means that in the case under scrutiny
	the Bridge estimator allows to obtain predictions for longer time periods, thanks to the slower growth of 
	its error rate. 
\end{enumerate}

\begin{figure}
	\centering
	\includegraphics[width = \textwidth]{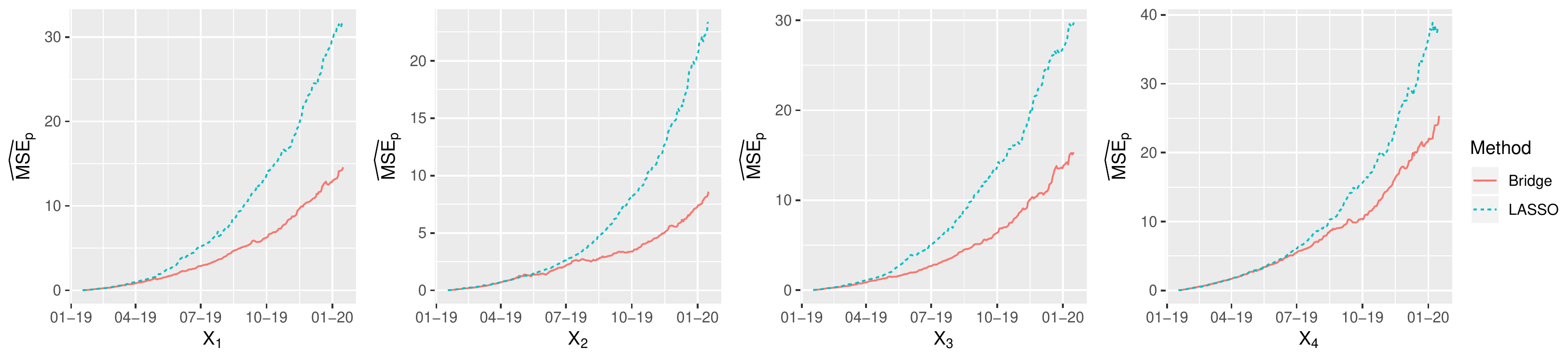}
	\caption{Comparison of the predictive MSE over time. }
	\label{fig:mse-time}
\end{figure}

\subsection{Comparison with disjoint estimation}
In \cite{suz}, the authors introduced the penalized estimator for ergodic diffusions \eqref{sde} defined as follows by
\begin{equation}\label{eq:est-non-joint}
	\hat{\alpha}^{(q_1)}_{n} \in \arg\min_{\theta \in \bar{\Theta}_1} Q^{(q_1)}_{1,n}(\alpha),
	\qquad 
	\hat{\beta}^{(q_2)}_{n} \in \arg\min_{\theta \in \bar{\Theta}_2} Q^{(q_2)}_{2,n}(\beta),
\end{equation}
where 
\begin{equation*}
Q^{(q_1)}_{1,n}(\alpha) =  (\alpha - \tilde \alpha_{n})'
\hat{G}_{1,n}   (\alpha - \tilde \alpha_{n})  + \sum_{i=1}^{p_1} \kappa_{i,n}^{1} | \alpha_i|^{q_1},
\end{equation*}
\begin{equation*}
Q^{(q_2)}_{2,n}(\beta) =  (\beta - \tilde \beta_{n})'
\hat{G}_{2,n}   (\beta - \tilde \beta_{n})  + \sum_{i=1}^{p_2} \kappa_{i,n}^{2} | \beta_i|^{q_2},
\end{equation*}
and $\hat{G}_{j,n}$ are $p_j \times p_j$ random matrices satisfying suitable regularity conditions, $(\tilde\alpha_n,\tilde\beta_n)'$ is an initial unpenalized estimator 
and $\kappa_{i,n}^{j}, j=1,2 $ represents suitable  adaptive weights (see \cite{suz} for details).
The main difference between this estimator and \eqref{eq:estsde} is that the former 
estimates each parameter group separately: in the following we refer to \eqref{eq:est-non-joint} as \emph{disjoint} estimator, while we call {\it joint} estimator \eqref{eq:estsde}.

In this section we compare the performances of the joint and disjoint estimators on a simple model. 
Consider the following SDE

\begin{equation}\label{eq:sdecomp}
\begin{pmatrix}
\mathrm d X_{1,t} \\
\mathrm d X_{2,t} 
\end{pmatrix}
= 
\begin{pmatrix}
-\alpha_{11} X_{1,t}^3 + \alpha_{12}(\sin X_{2,t} +2  )
\\ + \alpha_{21}(\cos X_{1,t} +2  )-\alpha_{22} X_{2,t}  
\end{pmatrix}
\mathrm d t + 
\begin{pmatrix}
\beta_{11} & 1 \\
1 & \beta_{22}
\end{pmatrix}
\begin{pmatrix}
\mathrm d W_{1,t} \\
\mathrm d W_{2,t} 
\end{pmatrix},
\end{equation}
$ 0 \leq t \leq T = 10 $, $(X_{1,0}, X_{2,0})' = (1,1)' $, where $ (\alpha_{11},  \alpha_{12},  \alpha_{21},  \alpha_{22})'  \in [0, 10]^4$
and $ (\beta_{11}, \beta_{22})'  \in [0, 10]^2$. This model 
is closely related to the one considered in \cite{uchad}. 
We simulated $N=10^3$ sample paths from this model, each with $n=10^3$ equally spaced data points, with true parameter value
 $(\alpha_{11}^*,  \alpha_{12}^*,  \alpha_{21}^*,  \alpha_{22}^*)' = (1, 0, 0, 1)', \,
 (\beta_{11}^*, \beta_{22}^*) '= (0,0)'$.
 
 We compared the performances of the joint and disjoint estimation techniques 
 for the Bridge estimator with $q_1 = q_2 = 0.9$, with tuning parameters 
 $\lambda_1 = \lambda_2 = 10$, $\delta_1 = \delta_2 = 2.5$, and QMLE as initial estimator. 
 
Table \ref{taba:joint-disjoint} shows the empirical selection probabilities for the zero parameters, while Table \ref{tabb:joint-disjoint} shows the empirical mean squared errors (up to the third decimal digit).

 \begin{table}[h]
 	\centering
 \begin{tabular}{ccccc}
 	\hline 
 	Sel. Prob. & $\alpha_{12}$ & $\alpha_{21}$ &   $\beta_{11}$ & $\beta_{22}$ \\
 	\hline 
 	Joint & 0.987    &  0.998    &  0.093   &   0.199 \\
 	Disjoint & 0.997   &   1.000  &    0.095  &    0.218   \\
 	\hline
 \end{tabular}
 \caption{Selection probabilities for the parameters with true value equal to zero}\label{taba:joint-disjoint}
\end{table}

\begin{table}[h]
	\centering
	\begin{tabular}{ccccccc}
		\hline 
		MSE & $\alpha_{11}$ & $\alpha_{12}$ & $\alpha_{21}$ & $\alpha_{22}$ &   $\beta_{11}$ & $\beta_{22}$ \\
		\hline 
		Joint & 0.913   &   0.004   &   0.000    &  0.664 &     7.824   &   1.901  \\
		Disjoint & 0.967   &   0.001  &    0.000    &  0.752   &  42.700   &   1.081  \\
		\hline
	\end{tabular}
	\caption{Mean square error for each parameter}\label{tabb:joint-disjoint}
\end{table}

We see that whereas the selection probabilities are practically equivalent up to some Monte Carlo sample error, the mean square errors can be significantly higher for the disjoint method. Therefore, the joint estimator turns out to have a better performance with respect to the disjoint one, at least for the SDE \eqref{eq:sdecomp}.

\section{Proofs}\label{sec6}

	\begin{proof}[Proof of Theorem \ref{th:cons}] For the proof of this theorem, we were inspired from the proof of Theorem 1 in \cite{suz}. 
		Let us start by observing that
		
		\begin{align*}
		0&\geq  \mathcal F_n(\hat\theta_n)- \mathcal F_n(\theta_0)\\
		&=(\hat\theta_n-\tilde{\theta}_n)'\hat G_n(\hat\theta_n-\tilde{\theta}_n)+\sum_{i=1}^m\sum_{j=1}^{p_i}\lambda_{n,j}^i|\hat\theta_{n,j}^i|^{q_i} \\
		&\quad -\left((\theta_0-\tilde{\theta}_n)'\hat G_n(\theta_0-\tilde{\theta}_n)+\sum_{i=1}^m\sum_{j=1}^{p_i}\lambda_{n,j}^i|\theta_{0,j}^i|^{q_i} \right)\\
		&=(\hat\theta_n-\theta_0)'\hat G_n(\hat\theta_n-\theta_0)+2(\hat\theta_n-\theta_0)'\hat G_n(\theta_0-\tilde\theta_n)+\sum_{i=1}^m\sum_{j=1}^{p_i}\lambda_{n,j}^i\left(|\hat\theta_{n,j}^i|^{q_i}-|\theta_{0,j}^i|^{q_i}\right)		\end{align*}
		Let $K_i:=\max_{1\leq j\leq p_i^0}| \theta_{0,j}|^{q_i-1},i=1,...,m.$ By exploiting the same arguments adopted in the proof of Theorem 1 in \cite{suz}, we can write down
		$$\sum_{j=1}^{p_i}\lambda_{n,j}^i\left(|\hat\theta_{n,j}^i|^{q_i}-|\theta_{0,j}^i|^{q_i}\right)\geq - p_i^0 K_i a_n^i|\hat\theta_n^i-\theta_0^i|,\quad i=1,...,m.$$
		Let $||\cdot||$ be a matrix norm. We get
		
			\begin{align*}
		0&\geq (\hat\theta_n-\theta_0)'\hat G_n(\hat\theta_n-\theta_0)+2(\hat\theta_n-\theta_0)'\hat G_n(\theta_0-\tilde\theta_n)\\
		&\quad-\sum_{i=1}^m p_i^0 K_i r_n^i  a_n^i |(r_n^i)^{-1}(\hat\theta_n^i-\theta_0^i)|\\
		&\geq (\hat\theta_n-\theta_0)'\hat G_n(\hat\theta_n-\theta_0)+2(\hat\theta_n-\theta_0)'\hat G_n(\theta_0-\tilde\theta_n)\\
		&\quad-\left(\sum_{i=1}^m p_i^0 K_i r_n^i  a_n^i\right)|A_n^{-1}(\hat\theta_n-\theta_0)|\\
		&\geq [A_n^{-1}(\hat\theta_n-\theta_0)]'\hat{\mathfrak D}_n[A_n^{-1}(\hat\theta_n-\theta_0)]+2[A_n^{-1}(\hat\theta_n-\theta_0)]'\hat{\mathfrak D}_n[A_n^{-1}(\theta_0-\tilde\theta_n)]\\
		&\quad-\left(\sum_{i=1}^m p_i^0 K_i r_n^i  a_n^i\right)|A_n^{-1}(\hat\theta_n-\theta_0)|.
		\end{align*}
		Let $ \rho_{\text{min}}(M)$  and $\rho_{\text{max}}(M)$ be  the minimum and maximum eigenvalue, respectively, of a matrix $M.$ We have
		\begin{align*}
		 [A_n^{-1}(\hat\theta_n-\theta_0)]'\hat{\mathfrak D}_n[A_n^{-1}(\hat\theta_n-\theta_0)]&\geq\rho_{\text{min}}(\hat{\mathfrak D}_n)|A_n^{-1}(\hat\theta_n-\theta_0)|^2\\
		 &\geq  ||\hat{\mathfrak D}_n^{-1}||^{-1}|A_n^{-1}(\hat\theta_n-\theta_0)|^2,
		\end{align*}
		where the last step follows from $||\hat{\mathfrak D}_n^{-1}||\geq\rho_{\text{max}}(\hat{\mathfrak D}_n^{-1})=1/\rho_{\text{min}}(\hat{\mathfrak D}_n).$ Furthermore,
			\begin{align*}
			[A_n^{-1}(\hat\theta_n-\theta_0)]'\hat{\mathfrak D}_n[A_n^{-1}(\theta_0-\tilde\theta_n)]&\geq -|[A_n^{-1}(\hat\theta_n-\theta_0)]'\hat{\mathfrak D}_n[A_n^{-1}(\theta_0-\tilde\theta_n)]|\\
			&\geq -|A_n^{-1}(\hat\theta_n-\theta_0)|\,|\hat{\mathfrak D}_n[A_n^{-1}(\theta_0-\tilde\theta_n)]|\\
			&\geq -|A_n^{-1}(\hat\theta_n-\theta_0)|\, ||\hat{\mathfrak D}_n||\, |A_n^{-1}(\tilde\theta_n-\theta_0)|.
				\end{align*}
		Hence, we have proved that	
		\begin{align*}
		0&\geq ||\hat{\mathfrak D}_n^{-1}||^{-1}|A_n^{-1}(\hat\theta_n-\theta_0)|^2-2 ||\hat{\mathfrak D}_n||(|A_n^{-1}(\hat\theta_n-\theta_0)|\, |A_n^{-1}(\tilde\theta_n-\theta_0)|)\\
		&\quad -\left(\sum_{i=1}^m p_i^0 K_i r_n^i  a_n^i\right)|A_n^{-1}(\hat\theta_n-\theta_0)|.
		\end{align*}
		Therefore, from the assumptions
		\begin{align}\label{eq:bound}
		|A_n^{-1}(\hat\theta_n-\theta_0)|\leq ||\hat{\mathfrak D}_n^{-1}||\left[2 ||\hat{\mathfrak D}_n||\, |A_n^{-1}(\tilde\theta_n-\theta_0)|+\sum_{i=1}^m p_i^0 K_i r_n^i  a_n^i\right]=O_p(1),
		\end{align}
		which concludes the proof.
		\end{proof}

	\begin{proof}[Proof of Theorem \ref{th:sc}]
		By taking into account the standard approach based on the Karush-Kuhn-Tucker (KKT) conditions, we are able to prove the selection consistency property of the Bridge-type estimator \eqref{eq:LASSOtypest}. Let us assume that $\hat\theta_{n,j}^i\neq 0$ for some $j=p_i^0+1,...,p_i.$   Let us note that
		$$\hat G_n=\left(\begin{matrix}
		\hat G^1_n\\
		\hat G^2_n\\
		\vdots\\
		\hat G^m_n
		\end{matrix}\right)$$
		where $\hat G^i_n$ is a $p_i\times \mathfrak p$ random matrix, for $i=1,2,...,m$. Furthermore		\begin{align}\label{kkt1}
		r_n^i\frac{\partial}{\partial\theta_j^i}\mathcal F_n(\theta)\Bigg|_{\theta=\hat\theta_n}=2r_n^i \hat G_n^i(j) A_nA_n^{-1}(\hat\theta_n-\tilde\theta_n) +r_n^i q_i \lambda_{n,j}^i|\hat \theta_{n,j}^i|^{q_i-1}\sgn(\hat \theta_{n,j}^i)=0,
		\end{align}
		 for $j=p_i^0+1,...,p_i$ and $i=1,...,m$.
	By $\hat G_n^i(j)$ we denote the $j$-th row of $\hat G_n^i$. From \eqref{kkt1}, one has
 
 \begin{align*}
 |2r_n^i\hat G_n^i(j) A_nA_n^{-1}(\hat\theta_n-\tilde\theta_n) |\,|r_n^i\hat \theta_{n,j}^i |^{1-q_i}=q_i   (r_n^i)^{2-q_i} \lambda_{n,j}^i\geq q_i   (r_n^i)^{2-q_i}  b_n^i.
 \end{align*}
 
 By Theorem \ref{th:cons} and the assumptions, we have that  $$\underbrace{ |2r_n^i\hat G_n^i(j) A_nA_n^{-1}(\hat\theta_n-\tilde\theta_n) |}_{O_p(1)}\,\underbrace{|r_n^i\hat \theta_{n,j}^i |^{1-q_i}}_{o_p(1)}=o_p(1),$$ 
 while $q_i   (r_n^i)^{2-q_i}  b_n^i\stackrel{p}{\longrightarrow}\infty$.
  Therefore, for any  for $j=p_i^0+1,...,p_i,$ it turns out that
 \begin{align*}
 P\left(\hat \theta_{n,j}^i\neq 0\right)\leq P\left( |2r_n^i\hat G_n^i(j) A_nA_n^{-1}(\hat\theta_n-\tilde\theta_n) |\,|r_n^i\hat \theta_{n,j}^i |^{1-q_i}\geq q_i   (r_n^i)^{2-q_i}  b_n^i\right)\longrightarrow 0,
 \end{align*}
 as $n\longrightarrow \infty.$
 \end{proof}

	\begin{proof}[Proof of Theorem \ref{th:an}] In order to simplify the reading of the proof we drop the dependence from n; then we set $\hat\theta:=\hat\theta_n, \tilde\theta:=\tilde\theta_n$ and $\hat G:=\hat G_n.$ We will use an approach similar to that developed in the proof of the Theorem 3 in \cite{suz}. 
		Let us rewrite $\hat G$ as a partitioned matrix 
		$$\hat G =\left(  \begin{matrix} 
	      \hat G^{11} &    \hat G^{12} &\cdots &\hat G^{1m}\\
	        \hat G^{21}   &      \hat G^{22} & \cdots &\hat G^{2m} \\
	        \vdots &\vdots&

\ddots & \vdots \\
   \hat G^{m1}   &      \hat G^{m2} & \cdots &\hat G^{mm} 
	   \end{matrix}	\right)$$  
	where the blocks are given by
	$$\hat G^{ij} =\left(  \begin{matrix} 
	      \hat G^{ij}_{\star\star} &   \hat G^{ij}_{\star\bullet} \\
	        \hat G^{ij}_{\bullet\star} &   \hat G^{ij}_{\bullet\bullet} \\
	   \end{matrix}	\right),\quad 1\leq i,j\leq m.$$  
	We observe that
	   \begin{align*}
	   \mathcal{F}_n(\theta)&=(\theta-\tilde{\theta})'\hat G(\theta-\tilde{\theta})+\sum_{i=1}^m\sum_{j=1}^{p_i}\lambda_{n,j}^i|\theta_j^i|^{q_i}\\
	   &=\sum_{i=1}^m (\theta^i-\tilde{\theta}^i)_\star'\hat G_{\star\star}^{ii}(\theta^i-\tilde{\theta}^i)_\star+\sum_{i=1}^m (\theta^i-\tilde{\theta}^i)_\bullet'\hat G_{\bullet\bullet}^{ii}(\theta^i-\tilde{\theta}^i)_\bullet\\
	   &\quad+2\sum_{i=1}^m (\theta^i-\tilde{\theta}^i)_\star'\hat G_{\star\bullet}^{ii}(\theta^i-\tilde{\theta}^i)_\bullet+2\sum_{i=1}^m \sum_{j>i}\left[ (\theta^i-\tilde{\theta}^i)_\star'\hat G_{\star\star}^{ij}(\theta^j-\tilde{\theta}^j)_\star\right.\\
	   &\quad\left.+ (\theta^i-\tilde{\theta}^i)_\star'\hat G_{\star\bullet}^{ij}(\theta^j-\tilde{\theta}^j)_\bullet+(\theta^i-\tilde{\theta}^i)_\bullet'\hat G_{\bullet\star}^{ij}(\theta^j-\tilde{\theta}^j)_\star+(\theta^i-\tilde{\theta}^i)_\bullet'\hat G_{\bullet\bullet}^{ij}(\theta^j-\tilde{\theta}^j)_\bullet\right]\\
	   &\quad+\sum_{i=1}^m\sum_{j=1}^{p_i^0}\lambda_{n,j}^i|\theta_j^i|^{q_i}+\sum_{i=1}^m\sum_{j=p_i^0+1}^{p_i}\lambda_{n,j}^i|\theta_j^i|^{q_i}
	   \end{align*}
	   By setting $\check \theta:=(\theta^1_\star,0,\theta^2_\star,0,..., \theta^m_\star, 0)'\in\mathbb R^{\mathfrak p},$ we have 
	     \begin{align*}
	   \mathcal{F}_n(\check\theta)
	   &=\sum_{i=1}^m (\theta^i-\tilde{\theta}^i)_\star'\hat G_{\star\star}^{ii}(\theta^i-\tilde{\theta}^i)_\star+\sum_{i=1}^m (\tilde{\theta}^i)_\bullet'\hat G_{\bullet\bullet}^{ii}\tilde{\theta}^i_\bullet-2\sum_{i=1}^m (\theta^i-\tilde{\theta}^i)_\star'\hat G_{\star\bullet}^{ii}\tilde{\theta}^i_\bullet\\
	   &\quad+2\sum_{i=1}^m \sum_{j>i}\left[ (\theta^i-\tilde{\theta}^i)_\star'\hat G_{\star\star}^{ij}(\theta^j-\tilde{\theta}^j)_\star- (\theta^i-\tilde{\theta}^i)_\star'\hat G_{\star\bullet}^{ij}\tilde{\theta}^j_\bullet-(\tilde{\theta}^i_\bullet)'\hat G_{\bullet\star}^{ij}(\theta^j-\tilde{\theta}^j)_\star+(\tilde{\theta}^i_\bullet)'\hat G_{\bullet\bullet}^{ij}\tilde{\theta}^j_\bullet\right]\\
	   &\quad+\sum_{i=1}^m\sum_{j=1}^{p_i^0}\lambda_{n,j}^i|\theta_j^i|^{q_i}
	   \end{align*}
	   
	   Let $B_n^i:=\{\min_{1\leq j\leq p_i^0}|\hat \theta_j^i|>0,\, \hat \theta^i_\bullet =0, \,\text{det}(\hat G^{ii}_{\star\star})>0\},$ by Theorem \ref{th:cons}-\ref{th:sc} follows $P(\cap_{i=1}^m B_n^i)\longrightarrow 1.$ 	 We observe that, if $\cap_{i=1}^m B_n^i$ holds, then $\mathcal F_n(\hat \theta)=\min_{\check \theta\in \mathbb R_0^{\mathfrak{p}}} \mathcal F_n(\check \theta),$ where $\mathbb R_0^{\mathfrak{p}}:=\{\theta\in\mathbb R^{\mathfrak{p}}: \theta_\bullet^i=0, i=1,...,m\}.$ This remark implies on $B_n^i$
	   \begin{align*}
	   0=\frac12\frac{\partial}{\partial \theta_\star^i}\mathcal F_n(\theta)\Bigg |_{\theta=\hat \theta}&=\hat G_{\star\star}^{ii}(\hat \theta^i-\tilde\theta^i)_\star-\hat G^{ii}_{\star\bullet}\tilde\theta^i_\bullet\\
	   &\quad+\sum_{j>i}\left[\hat G_{\star\star}^{ij}(\hat \theta^j-\tilde\theta^j)_\star-\hat G^{ij}_{\star\bullet}\tilde\theta^j_\bullet \right] +Z(\hat \theta^i)
	   \end{align*}  
	   where $Z(\hat \theta^i):=(\frac12 q_i\lambda_{n,1}^i |\hat \theta_1^i|^{q_i-1}\sgn(\hat \theta_1^i),...,\frac12 q_i\lambda_{n,p_i^0}^i |\hat \theta_{p_i^0}^i|^{q_i-1}\sgn(\hat \theta_{p_i^0}^i))'.$ Therefore
	   \begin{align*}
	   (\hat \theta^i-\theta_0^i)_\star&=(\tilde \theta^i-\theta_0^i)_\star+(\hat G^{ii}_{\star\star})^{-1}\hat G^{ii}_{\star\bullet}\tilde\theta^i_\bullet\\
	   &\quad-\sum_{j>i}\left[(\hat G^{ii}_{\star\star})^{-1}\hat G_{\star\star}^{ij}(\hat \theta^j-\tilde\theta^j)_\star-(\hat G^{ii}_{\star\star})^{-1}\hat G^{ij}_{\star\bullet}\tilde\theta^j_\bullet \right] -(\hat G^{ii}_{\star\star})^{-1}Z(\hat \theta^i).
	   \end{align*}
	   
	   Let $\hat{\mathfrak G}_i:=({\bf I}_{p_i^0}\,\, \frac{1}{(r_n^i)^2}(\hat G^{ii}_{\star\star})^{-1}(r_n^i)^2 \hat G^{ii}_{\star\bullet})\stackrel{p}{\longrightarrow } \mathfrak G_i.$ Hence
	   \begin{align*}
	   &\frac{1}{r^i_n}(\hat\theta^i-\theta^i_0)_\star-  \mathfrak G_i \left\{\frac{1}{r^i_n}(\tilde\theta^i-\theta^i_0)\right\}\\
	   &={\bf 1}_{B_n^i}\Bigg\{\frac{1}{r^i_n}(\tilde \theta^i-\theta_0^i)_\star+\frac{1}{r^i_n}(\hat G^{ii}_{\star\star})^{-1}\hat G^{ii}_{\star\bullet}\tilde\theta^i_\bullet-\sum_{j>i}\left[\frac{1}{r^i_n}(\hat G^{ii}_{\star\star})^{-1}\hat G_{\star\star}^{ij}(\hat \theta^j-\tilde\theta^j)_\star-\frac{1}{r^i_n}(\hat G^{ii}_{\star\star})^{-1}\hat G^{ij}_{\star\bullet}\tilde\theta^j_\bullet \right] \\
	   &\quad-\frac{1}{r^i_n}(\hat G^{ii}_{\star\star})^{-1}Z(\hat \theta^i)-  \mathfrak G_i \left\{\frac{1}{r^i_n}(\tilde\theta^i-\theta^i_0)\right\}\Bigg\}+{\bf 1}_{(B_n^i)^c}\left\{\frac{1}{r^i_n}(\hat\theta^i-\theta^i_0)_\star-  \mathfrak G_i \left\{\frac{1}{r^i_n}(\tilde\theta^i-\theta^i_0)\right\}\right\}\\
	   &={\bf 1}_{B_n^i}\Bigg\{   (\hat{\mathfrak G}_i-\mathfrak G_i)\left\{\frac{1}{r^i_n}(\tilde \theta^i-\theta_0^i)\right\} -\sum_{j>i}\left[\frac{1}{r^i_n}(\hat G^{ii}_{\star\star})^{-1}\hat G_{\star\star}^{ij}(\hat \theta^j-\tilde\theta^j)_\star-\frac{1}{r^i_n}(\hat G^{ii}_{\star\star})^{-1}\hat G^{ij}_{\star\bullet}\tilde\theta^j_\bullet \right]\\
	   &\quad -\frac{1}{r^i_n}(\hat G^{ii}_{\star\star})^{-1}Z(\hat \theta^i) \Bigg\}
	 +{\bf 1}_{(B_n^i)^c}\left\{\frac{1}{r^i_n}(\hat\theta^i-\theta^i_0)_\star-  \mathfrak G_i \left\{\frac{1}{r^i_n}(\tilde\theta^i-\theta^i_0)\right\}\right\}\\
	  &={\bf 1}_{B_n^i}\left\{   (\hat{\mathfrak G}_i-\mathfrak G_i)\left\{\frac{1}{r^i_n}(\tilde \theta^i-\theta_0^i)\right\}+o_p(1) \right\}	 +{\bf 1}_{(B_n^i)^c}\left\{\frac{1}{r^i_n}(\hat\theta^i-\theta^i_0)_\star-  \mathfrak G_i \left\{\frac{1}{r^i_n}(\tilde\theta^i-\theta^i_0)\right\}\right\},
	   \end{align*}
	   where the last step holds because:
	   \begin{align*}
	  &\frac{1}{(r_n^i)^2}(\hat G^{ii}_{\star\star})^{-1}r_n^ir_n^j\hat G_{\star\star}^{ij}\frac{1}{r_n^j}(\hat \theta^j-\tilde\theta^j)_\star\,{\bf 1}_{B_n^i}= o_p(1) O_p(1)=o_p(1);\\
	 & \frac{1}{(r^i_n)^2}(\hat G^{ii}_{\star\star})^{-1}r_n^ir_n^j\hat G^{ij}_{\star\bullet}\frac{1}{r_n^j}\tilde\theta^j_\bullet\,{\bf 1}_{B_n^i}=o_p(1);\\
	 	 & \frac{1}{(r^i_n)^2}(\hat G^{ii}_{\star\star})^{-1}r_n^iZ(\hat \theta^i)\,{\bf 1}_{B_n^i}=o_p(1).
	 \end{align*}
	 
	 Finally,
	 $$\frac{1}{r^i_n}(\hat\theta^i-\theta^i_0)_\star-  \mathfrak G_i \left\{\frac{1}{r^i_n}(\tilde\theta^i-\theta^i_0)\right\}\stackrel{p}{\longrightarrow }0,   $$
	and then the result \eqref{oraclegen} holds.
	
   By adding the assumption A3,	\eqref{oracleasnorm} is a trivial consequence of \eqref{oraclegen}.  Furthermore, if $G=\Gamma=\text{diag}(\Gamma^{11},\Gamma^{22},...,\Gamma^{mm}),$ we get
	$$\mathfrak I=\text{diag}(\mathfrak I^{11},\mathfrak I^{22},...,\mathfrak I^{mm})$$
	and
	$$\mathfrak G\,\mathfrak I\,\mathfrak G'=\text{diag}(\mathfrak G_1\,\mathfrak I^{11}\,\mathfrak G_1',\mathfrak G_2\,\mathfrak I^{22}\,\mathfrak G_2',...,\mathfrak G_m\,\mathfrak I^{mm}\,\mathfrak G_m'),$$
	where $\mathfrak I^{ii}:=(\Gamma^{ii})^{-1},i=1,2,...,.m.$	
	 By exploiting the blockwise inversion of $\Gamma^{ii},$ we recall that
	\begin{equation}\label{eq:matrix}
	\mathfrak I^{ii}= \left( \begin{matrix} 
	       \mathfrak I_{\star\star}^{ii} & -   \mathfrak I_{\star\star}^{ii} \Gamma^{ii}_{\star\bullet} (\Gamma_{\bullet\bullet}^{ii})^{-1}\\\\
	      -   (\Gamma_{\bullet\bullet}^{ii})^{-1}\Gamma^{ii}_{\bullet\star}  \mathfrak I_{\star\star}^{ii} &  ( \Gamma^{ii}_{\bullet\bullet})^{-1}+( \Gamma^{ii}_{\bullet\bullet})^{-1}\Gamma^{ii}_{\bullet\star}   \mathfrak I_{\star\star}^{ii} \Gamma^{ii}_{\star\bullet} (\Gamma_{\bullet\bullet}^{ii})^{-1} \\
	   \end{matrix}\right)\end{equation}
	   where
	      \begin{equation} \label{eq:matrix2} \mathfrak I_{\star\star}^{ii} =(\Gamma^{ii}_{\star\star}-\Gamma_{\star\bullet}^{ii}(\Gamma_{\bullet\bullet}^{ii})^{-1}\Gamma_{\bullet\star}^{ii})^{-1}.\end{equation}
	      By taking into account \eqref{eq:matrix} and \eqref{eq:matrix2}, we obtain
	\begin{align*}
	\mathfrak G_i\, \mathfrak I^{ii}\,\mathfrak G_i'
	 &  = \mathfrak I_{\star\star}^{ii}- ( \Gamma^{ii}_{\star\star})^{-1} \Gamma^{ii}_{\star\bullet}(\Gamma_{\bullet\bullet}^{ii})^{-1}\Gamma^{ii}_{\bullet\star}  \mathfrak I_{\star\star}^{ii} - \mathfrak I_{\star\star}^{ii} \Gamma^{ii}_{\star\bullet} (\Gamma_{\bullet\bullet}^{ii})^{-1}   \Gamma^{ii}_{\bullet\star} ( \Gamma^{ii}_{\star\star})^{-1}\\
	   &\quad+ ( \Gamma^{ii}_{\star\star})^{-1} \Gamma^{ii}_{\star\bullet}(\Gamma_{\bullet\bullet}^{ii})^{-1}\Gamma^{ii}_{\bullet\star}  ( \Gamma^{ii}_{\star\star})^{-1}+ ( \Gamma^{ii}_{\star\star})^{-1} \Gamma^{ii}_{\star\bullet}(\Gamma_{\bullet\bullet}^{ii})^{-1}\Gamma^{ii}_{\bullet\star}  \mathfrak I_{\star\star}^{ii} \Gamma^{ii}_{\star\bullet}(\Gamma_{\bullet\bullet}^{ii})^{-1}\Gamma^{ii}_{\bullet\star} ( \Gamma^{ii}_{\star\star})^{-1} \\
	   &= \mathfrak I_{\star\star}^{ii}\left[\Gamma^{ii}_{\star	\star}- \Gamma^{ii}_{\star\bullet} (\Gamma_{\bullet\bullet}^{ii})^{-1}   \Gamma^{ii}_{\bullet\star} \right]( \Gamma^{ii}_{\star\star})^{-1}+ ( \Gamma^{ii}_{\star\star})^{-1} \Gamma^{ii}_{\star\bullet}(\Gamma_{\bullet\bullet}^{ii})^{-1}\Gamma^{ii}_{\bullet\star}  ( \Gamma^{ii}_{\star\star})^{-1}\\
	   &\quad-( \Gamma^{ii}_{\star\star})^{-1} \Gamma^{ii}_{\star\bullet}(\Gamma_{\bullet\bullet}^{ii})^{-1}\Gamma^{ii}_{\bullet\star}\mathfrak I_{\star\star}^{ii}\left[\Gamma^{ii}_{\star	\star}- \Gamma^{ii}_{\star\bullet} (\Gamma_{\bullet\bullet}^{ii})^{-1}   \Gamma^{ii}_{\bullet\star} \right]( \Gamma^{ii}_{\star\star})^{-1}\\
	   &=(\Gamma^{ii}_{\star	\star})^{-1},
	\end{align*}
which concludes the proof.	
	\end{proof}

\begin{proof}[Proof of Theorem \ref{th:sde}.]

The consistency, the selection consistency and the asymptotic normality are consequences of Theorem \ref{th:cons}-\ref{th:an}. From \eqref{eq:bound} and the Cauchy-Scharwz inequality, we derive the following bound
	\begin{align*}
	\mathbb{E} |A_n^{-1}(\hat\theta_n-\theta_0)|^q &\leq 2^{2q-1}\mathbb{E}\left[||\hat{\mathfrak D}_n^{-1}||^q\, ||\hat{\mathfrak D}_n||^q\, |A_n^{-1}(\tilde\theta_n^{QL}-\theta_0)|^q\right] +2^{(q-1)(m-1)}\sum_{i=1}^m\mathbb{E} \left[ p_i^0 K_i r_n^i  a_n^i\right]^q\\
	&	\leq  2^{2q-1} \sqrt{\mathbb{E}||\hat{\mathfrak D}_n^{-1}||^{2q}} \left(\mathbb{E} ||\hat{\mathfrak D}_n||^{4q}\right)^{1/4}\, \left(\mathbb E|A_n^{-1}(\tilde\theta_n^{QL}-\theta_0)|^{4q} \right)^{1/4}\\
	&\quad+2^{(q-1)(m-1)}\sum_{i=1}^m\mathbb{E} \left[ p_i^0 K_i r_n^i  a_n^i\right]^q.
	\end{align*}
	
	From the assumptions, the polynomial-type large deviation result (25) and Proposition 1 in \cite{yospol}, we obtain the uniformly $L^q$-boundedness  of the estimator.
\end{proof}

\end{document}